\PassOptionsToPackage{}{algpseudocode}
\PassOptionsToPackage{usenames,dvipsnames,table}{xcolor}
\documentclass[a4paper, 10pt, USenglish, reqno]{amsart}

\input{TeX/Preamble.tex}

\newcommand\TheKeywords{%
	Nonsmooth nonconvex optimization,
	Proximal alternating linearized minimization,
	Bregman distance,
     Multi-block relative smoothness,
     KL inequality,
     Orthogonal nonnegative matrix factorization%
}
\newcommand\TheSubjclass{%
	90C06, 
	90C25, 
	90C26, 
	49J52, 
	49J53.
}
\newcommand\TheTitle{%
	Multi-block Bregman proximal alternating linearized minimization and its application to  orthogonal nonnegative matrix factorization%
}
\newcommand\TheShortTitle{%
	Multi-block Bregman proximal alternating linearized minimization%
}
\newcommand\TheShortAuthor{%
	M. Ahookhosh, L.T.K. Hien, N. Gillis, and P. Patrinos%
}
\newcommand\TheFunding{%
	MA and PP acknowledge the support by the \emph{Research Foundation Flanders (FWO)} research projects G086518N and G086318N;
	\emph{Research Council KU Leuven} C1 project No. C14/18/068;
	\emph{Fonds de la Recherche Scientifique - FNRS and the Fonds Wetenschappelijk Onderzoek - Vlaanderen} (FWO) under EOS project no 30468160 (SeLMA).
	NG  also acknowledges the support by the European Research Council (ERC starting grant no 679515).
}

\ifsiam
	
	\headers{\TheShortTitle}{\TheShortAuthor}
	
	\title{%
		\TheTitle%
		\thanks{%
			Submitted to the editors \today.%
			\funding{\TheFunding}%
		}%
	}

	\author{%
		Masoud Ahookhosh%
		\thanks{%
			\TheAddressKU.
			{\tt
				\{%
					\href{mailto:masoud.ahookhosh@esat.kuleuven.be}{masoud.ahookhosh},%
					\href{mailto:panos.patrinos@esat.kuleuven.be}{panos.patrinos}%
				\}\href{mailto:masoud.ahookhosh@esat.kuleuven.be,panos.patrinos@esat.kuleuven.be}{@esat.kuleuven.be}%
			}%
		}%
		\and
		Le Thi Khanh Hien%
		\thanks{%
			Department of Mathematics and Operational Research, Faculté polytechnique, Université de Mons.
Rue de Houdain 9,
7000 Mons,
Belgium.
			{\tt
				\{%
					\href{mailto:ThiKhanhHien.LE@umons.ac.be>}{ThiKhanhHien.LE},%
					\href{mailto:nicolas.gillis@umons.ac.be}{nicolas.gillis}%
				\}%
				\href{mailto:ThiKhanhHien.LE@umons.ac.be,nicolas.gillis@umons.ac.be}{@umons.ac.be}%
			}%
		}%
		\and
		Nicolas Gillis\texorpdfstring{\footnotemark[3]}{}%
		\and
		Panagiotis Patrinos\texorpdfstring{\footnotemark[2]}{}%
	}%

\else

	\title[\TheShortTitle]{\TheTitle}
	\author[\TheShortAuthor]{%
		Masoud Ahookhosh\textsuperscript{1},\ 
		Le Thi Khanh Hien\textsuperscript{2}.\ 
		Nicolas Gillis\textsuperscript{2},\ and\ 
		Panagiotis Patrinos\textsuperscript{1}%
	}
	\thanks{\hspace*{-\parindent}\textsuperscript{\small 1}\TheAddressKU\\\textsuperscript{\small 2}\TheAddressMons\\\TheFunding}
\fi

\begin{document}
	
	\ifams\else
		\maketitle
	\fi
	\begin{abstract}
		We introduce and analyze \refBPALM[] and  \refaBPALM[], two multi-block proximal alternating linearized minimization algorithms using Bregman distances for solving structured nonconvex problems. The objective function is the sum of a multi-block relatively smooth function (i.e., relatively smooth by fixing all the blocks except one) and block separable (nonsmooth) nonconvex functions. It turns out that the sequences generated by our algorithms are subsequentially convergent to critical points of the objective function, while they are globally convergent under KL inequality assumption. Further, the rate of convergence is further analyzed for functions satisfying the {\L}ojasiewicz's gradient inequality. We apply this framework to orthogonal nonnegative matrix factorization (ONMF) that satisfies all of our assumptions and the related subproblems are solved in closed forms, where some preliminary numerical results is reported. 
	\end{abstract}
	
	\ifams
		\maketitle
	\else
	\vspace{-2mm}
		\vspace{-2mm}
	\fi


\vspace{-3mm}
	\section{Introduction}
		
Consider the structured nonsmooth nonconvex minimization problem
\begin{equation}\label{eq:P}
	\minimize_{\bm x=(x_1,\ldots,x_N)\in\R^{\sum_i n_i}}~~\varphi(\bm x)\equiv f(\bm x)+\sum_{i=1}^N g_i(x_i),
\end{equation}
where we will systematically assume the following hypotheses (see \Cref{sec:algBPALM} for details):
\begin{ass}[requirements for composite minimization \eqref{eq:P}]\label{ass:basic:fgh}~
\begin{enumeratass}
	\item\label{ass:basic:g}%
		\(\func {g_i}{\R^{n_i}}{\Rinf\coloneqq\R\cup\set\infty}\) is proper and lower semicontinuous (lsc);
	\item\label{ass:basic:f}%
		\(\func{f}{\R^{n}}{\Rinf}\) is $\C^1(\interior\dom h)$ and \DEF{$(L_1,\ldots,L_N)$-smooth relative to $h$}; here $n=\sum_{i=1}^N n_i$;
	\item\label{ass:basic:h}%
		\(\func{h}{\R^{n}}{\Rinf}\) is \DEF{multi-block} strictly convex, \(1\)-coercive and essentially smooth;
	\item\label{ass:basic:argmin}%
		 $\varphi$ has a nonempty set of minimizers, i.e., \(\argmin\varphi\neq\emptyset\), and $\dom \varphi\subseteq \interior\dom h$;
	\item 
		the first-order oracles of $f$, $g_i\ (i=1,\ldots,N)$, and $h$ are available.
\end{enumeratass}
\end{ass}

Although, the problem \eqref{eq:P} has a simple structure, it covers a broad range of optimization problems arising in signal and image processing, statistical and machine learning, control and system identification. Consequently, needless to say, there is a huge number of algorithmic studies around solving the optimization problems of the form \eqref{eq:P}. Among all of such methodologies, we are interested in the class of \DEF{alternating minimization} algorithms such as block coordinate descent \cite{beck2015cyclic,beck2013convergence,latafat2019block,nesterov2012efficiency,razaviyayn2013unified,richtarik2014iteration,tseng2001convergence,tseng2009coordinate}, block coordinate \cite{combettes2015stochastic,fercoq2019coordinate,latafat2019new}, and Gauss-Seidel methods \cite{auslender1976optimisation,bertsekas1989parallel,grippo2000convergence}, which assumes that all blocks are fixed except one and solves the corresponding auxiliary problem with respect to this block, update the latter block, and continue with the others. In particular, the proximal alternating minimization has  received much attention in the last few years; see for example \cite{attouch2008alternating,attouch2007new,attouch2006inertia,attouch2010proximal,
attouch2013convergence,beck2016alternating}. Recently, the proximal alternating linearized minimization and its variation has been developed to handle \eqref{eq:P}; see for example \cite{bolte2014proximal,pock2016inertial,shefi2016rate}.

Traditionally, the Lipschitz (H\"older) continuity of partial gradients of $f$ in \eqref{eq:P} is a necessary tool for providing the convergence analysis of optimization algorithms; see, e.g., \cite{bolte2014proximal,pock2016inertial}. It is, however, well-known that it is not the Lipschitz (H\"older) continuity of gradients playing a key role in such analysis, but one of its consequence: an upper estimation of $f$ including a Bregman distance called \DEF{descent lemma} ; cf. \cite{bauschke2016descent,lu2018relatively}. This idea is central to convergence analysis of many optimization schemes requiring such an upper estimation; see, e.g., \cite{ahookhosh2019bregman,bauschke2019linear,bauschke2016descent,bolte2018first,teboulle2018simplified,hanzely2018fastest,hanzely2018accelerated,lu2018relatively}. 
In this paper, we propose a multi-block extension of the descent lemma given in \cite{bauschke2016descent,lu2018relatively} and propose a Bregman proximal alternating linearized minimization (\refBPALM[]) algorithm and its adaptive version (\refaBPALM[]) for \eqref{eq:P}.


		\subsection{Contribution}
			 Our contribution is summarized as follows:
\begin{enumerate}[wide, labelwidth=!, labelindent=0pt]
    \item (\DEF{Bregman proximal alternating linearized minimization}) We introduce \refBPALM[], a multi-block generalization of the proximal alternating linearized minimization (PALM) \cite{bolte2014proximal} using Bregman distances, and its adaptive version (\refaBPALM[]). To do so, we extend the notion of relative smoothness \cite{bauschke2016descent,lu2018relatively} to its multi-block counterpart to support a structured problem of the form \eqref{eq:P}. Owing to multi-block relative smoothness of $f$, unlike PALM, our algorithm does not need to know the local Lipschitz moduli of partial gradients $\nabla_i f$ ($i=1,\ldots,N$) and their lower and upper bounds, which are hard to provide in practice. 
    \item (\DEF{Efficient framework for ONMF}) Exploiting a suitable kernel for Bregman distance, it turns out that the objective of orthogonal nonnegative matrix factorization (ONMF) is multi-block relatively smooth, and the subproblems of our algorithms are solved in closed forms making them suitable for large-scale data analysis problems. To the best of our knowledge, \refBPALM[] and \refaBPALM[] are the first algorithms with rigorous convergence theory for ONMF.
\end{enumerate}

\subsection{Related works}	Closly related to our framework, there are two papers \cite{li2019provable,wang2018block}. However, we notice that \cite{wang2018block} uses a sum separable kernel function which is just a special case of our multi-block kernel functions (see \Cref{exa:popularKerFunc}), and the paper only provides a limited convergence theory. Regarding \cite{li2019provable}, an algorithm (named B-PALM) proposed that is just a special case of our \refBPALM[] when $N=2$, $g_1=g_2=0$, and $f\in\mathcal{C}^2$, which restrict its applications. We stress that involving the block separable nonsmooth nonconvex functions $g_i$ and considering $N>2$ make our analysis different from those of \cite{li2019provable}.

\subsection{Organization}
This paper has four sections, besides this introductory section.
In \Cref{sec:algBPALM}, we introduce the notion of multi-block relative smoothness, and verify the fundamental properties of Bregman proximal alternating linearized mapping. 
In \Cref{sec:convAnalysis}, we introduce \refBPALM[]  and  \refaBPALM[] and investigate their convergence analysis.
In \Cref{sec:appONMF}, we show that ONMF satisfies our assumptions, the related subproblems are solved in closed forms, and report our numerical results. Finally, \Cref{sec:conclusion} delivers some conclusions.
		\subsection{Notation}
			We denote by $\Rinf\coloneqq\R\cup\set{\infty}$ the extended-real line. For the identity matrix $I_n$, we set $U_i\in\R^{n\times n_i}$ such that
 $I_n=(U_1,\ldots,U_N)\in\R^{n\times n}$. 
The open ball of radius \(r\geq 0\) centered in \(x\in\R^p\) is denoted as \(\ball xr\).
The set of cluster points of \(\seq{x^k}\) is denoted as \(\omega(x^0)\).
A function $\func{f}{\R^p}{\Rinf}$ is \DEF{proper} if $f>-\infty$ and $f\not\equiv\infty$, in which case its \DEF{domain} is defined as the set
$\dom f\coloneqq\set{x\in\R^p}[f(x)<\infty]$. For \(\alpha\in\R\),
\(
	[f\leq\alpha]
{}\coloneqq{}
	\set{x\in\R^p}[f(x)\leq\alpha]
\)
is the \DEF{\(\alpha\)-(sub)level set} of \(f\); \([f\geq\alpha]\) and \([f=\alpha]\) are defined similarly.
We say that \(f\) is \DEF{level bounded} if \([f\leq\alpha]\) is bounded for all \(\alpha\in\R\).
A vector $v\in\partial f(x)$ is a \DEF{subgradient} of $f$ at $x$, and the set of all such vectors is called the \DEF{subdifferential} $\partial f(x)$ \cite[Definition 8.3]{rockafellar2011variational}, i.e. 
\begin{align*}
	\partial f(x)
{}={} &
	\set{v\in\R^p}[
		\exists\seq{x^k,v^k}~\text{s.t.}~ x^k\to x,~f(x^k)\to f(x),~
		\widehat\partial f(x^k)\ni v^k\to v
	],
\shortintertext{%
	and $\widehat\partial f(x)$ is the set of \DEF{regular subgradients} of $f$ at $x$, namely
}
	\widehat\partial f(x)
{}={} &
	\set{v\in\R^p}[
		f(z)\geq f(x){}+{}\innprod{v}{z-x}{}+{}o(\|z-x\|),~
		\forall z\in\R^p\vphantom{\seq{x^k}}
	].
\end{align*}

	\section{Multi-block Bregman proximal alternating linearized mapping}\label{sec:algBPALM}
		We first establish the notion multi-block relative smoothness, which is an extension of the relative smoothness \cite{bauschke2016descent,lu2018relatively} for problems of the form \eqref{eq:P}. We then introduce Bregman alternating linearized mapping and study some of its basic properties. For notation clarity, we will use bold lower-case letters (e.g., $\bm x$, $\bm y$, $\bm z$) for vectors in $\R^{\sum n_i}$ and use normal lower-case letters (e.g., $z$, $x_i$, $y_i$) for vectors in $\R^{n_i}$.

In order to extend the definition of Bregman distances for the multi-block problem \eqref{eq:P}, we first need to introduce the notion of \DEF{multi-block kernel} functions, which will coincide with the standard one (cf. \cite[Definition 2.1]{ahookhosh2019bregman}) if $N=1$.

\begin{defin}[multi-block convexity and kernel function]\label{def:kernel}%
	Let \(\func{h}{\R^n}{\Rinf}\) be a proper and lsc function with \(\interior\dom h\neq\emptyset\) and such that  $h\in\C^1(\interior\dom h)$.
	For a fixed vector $\bm x\in\R^n$, we define the function \(\func{h_{\bm x}^i}{\R^{n_i}}{\Rinf}\) given by 
	\begin{equation}\label{eq:hxi}
		 h_{\bm x}^i( z):=h(\bm x+U_i(z-x_i)).
	\end{equation}
	Then, we say that \(h\) is
	\begin{enumerate}
	\item 
	     \DEF{multi-block (strongly/strictly) convex} if the function $h_{\bm x}^i(\cdot)$ is (strongly/strictly) convex for all $\bm x\in \dom h$ and $i=1,\ldots,N$;
	\item 
	    \DEF{multi-block locally strongly convex} around $\bm x^\star= (x_1^\star,\ldots,x_N^\star)$ if, for $i=1,\ldots,N$, there exists $\delta>0$ and $\sigma_h^i>0$ such that
        \[
            h_{\bm x}^i(x_i)\geq h_{\bm x}^i(y_i)+\innprod{\nabla_i h(\bm y)}{x_i-y_i}+\tfrac{\sigma_h^i}{2}\|x_i-y_i\|^2\quad \forall \bm x,\bm y\in \ball{\bm x^\star}{\delta};
        \]
	\item
		a \DEF{multi-block kernel function} if $h$ is multi-block convex and $h_{\bm x}^i(\cdot)$ is $1$-coercive for all $\bm x\in \dom h$ and $i=1,\ldots,N$, i.e., $\lim_{\|z\|\to\infty}\tfrac{h_{\bm x}^i(z)}{\|z\|}=\infty$;
	\item
		\DEF{multi-block essentially smooth}, if for every sequence $\seq{\bm x^k}\subseteq\interior\dom h$ converging to a boundary point of $\dom h$, we have $\|\nabla_i h(\bm x^k)\|\to\infty$  for all $i=1,\ldots,N$;
	\item
		of \DEF{multi block Legendre} type if it is multi-block essentially smooth and multi-block strictly convex.
	\end{enumerate}
\end{defin}

\begin{es}[popular kernel functions]
\label{exa:popularKerFunc}
    There are many kernel functions satisfying \Cref{def:kernel}. For example, for $N=1$, energy, Boltzmann-Shannon entropy, Fermi-Dirac entropy (cf. \cite[Example 2.3]{bauschke2018regularizing}) and several examples in \cite[Section 2]{lu2018relatively}; and for $N=2$ see two examples in \cite[Section 2]{li2019provable}. Two important classes of multi-block kernels are \DEF{sum separable kernels}, i.e.,
    $
        h(x_1,\ldots,x_N)=h_1(x_1)+\ldots+h_N(x_N)
   $,
    and \DEF{product separable kernels}, i.e.,
    $
        h(x_1,\ldots,x_N)=h_1(x_1)\times\ldots\times h_N(x_N)
    $,
    see such a kernel for ONMF in \Cref{pro:relSmoothNMF0}.
\end{es}

We now give the definition of \DEF{Bregman distances} (cf. \cite{bregman1967relaxation}) for multi-block kernels.
 
\begin{defin}[Bregman distance]\label{def:BregDist}
	For a kernel function \(h\), the \DEF{Bregman distance} $\func{\D}{\R^n\times\R^n}{\Rinf}$ is given by
	\begin{equation}\label{eq:bregman}
		\D(\bm y,\bm x)
	{}\coloneqq{}
		\begin{ifcases}
			h(\bm y)-h(\bm x)-\innprod{\nabla h(\bm x)}{\bm y-\bm x} & \bm x\in\interior\dom h
		\\
			\infty\otherwise.
		\end{ifcases}
	\end{equation}
\end{defin}

Fixing all blocks except the $i$-th one, the Bregman distance with respect to this block is given by 
\begin{align*}
    \D(\bm x+U_i(y_i-x_i),\bm x)&=h(\bm x+U_i(y_i-x_i))-h(\bm x)-\langle\nabla h(\bm x),U_i(y_i-x_i)\rangle\\
    &=h_{\bm x}^i(y_i)-h_{\bm x}^i(x_i)-\langle\nabla_i h(\bm x),y_i-x_i\rangle,
\end{align*}
which measures the proximity between $\bm y$ and $\bm x$ with respect to the $i$-th block of variables.
Moreover, the kernel \(h\) is multi-block convex if and only if \(\D(\bm x+U_i(y_i-x_i),\bm x)\geq 0\) for all \(\bm y\in\dom h\) and \(\bm x\in\interior \dom h\) and $i=1,\ldots,N$. 
Note that if \(h\) is multi-block strictly convex, then \(\D(\bm x+U_i(y_i-x_i),\bm x)= 0\) ($i=1,\ldots,N$) if and only if \(x_i=y_i\). 

We are now in a position to present the notion of \DEF{multi-block relative smoothness}, which is the central tool for our analysis in \Cref{sec:convAnalysis}. 

\begin{defin}[multi-block relative smoothness]
\label{def:mbRelSmooth}
Let \(\func{h}{\R^n}{\Rinf}\) be a multi-block kernel and let \(\func{f}{\R^n}{\Rinf}\) be a proper and lower semicontinuous function. 
If there exists $L_i\geq 0$ ($i=1,\ldots,N$) such that the functions $\func{\phi_{\bm x}^i}{\R^{n_i}}{\Rinf}$ given by 
\begin{align*}
    \phi_{\bm x}^i(z):=L_ih(\bm x+U_i(z-x_i))-f(\bm x+U_i(z-x_i))
\end{align*}
are convex for all $\bm x\in \dom h$ and $i=1,\ldots,N$, then, $f$ is called \DEF{$(L_1,\ldots,L_N)$-smooth relative to $h$}.
\end{defin}

Note that if $N=1$, the multi-block relative smoothness is reduced to standard relative smoothness, which was introduced only recently in \cite{bauschke2016descent,lu2018relatively}. In this case, if \(f\) is \(L\)-Lipschitz continuous, then both \(\nicefrac{L}{2}\|\cdot\|^2-f\) and \(f-\nicefrac{L}{2}\|\cdot\|^2\) are convex, i.e., the relative smoothness of \(f\) generalizes the notions of Lipschitz continuity using Bregman distances. If $N=2$, this definition will be reduced to the relative bi-smoothness given in \cite{li2019provable} for $h, f\in\mathcal{C}^2$.

We next characterize the notion of multi-block relative smoothness. 

\begin{prop}[characterization of multi-block relative smoothness] 
\label{fac:relSmoothEqvi}
    Let \(\func{h}{\R^n}{\Rinf}\) be a multi-block kernel and let \(\func{f}{\R^n}{\Rinf}\) be a proper 
    lower semicontinuous function and \(f\in\mathcal{C}^1\). Then, the following statements are equivalent:
    \begin{enumerateq}
    \item \label{fac:relSmoothEqvi1}
    $(L_1,\ldots,L_N)$-smooth relative to $h$; 
    \item \label{fac:relSmoothEqvi2}
    for all \((\bm x,\bm y)\in\interior\dom h\times \interior\dom h\) and $i=1,\ldots,N$, 
    \begin{equation}\label{eq:upperIneq}
        f(\bm x+U_i(y_i-x_i)) \leq f(\bm x)+\innprod{\nabla_i f(\bm x)}{y_i-x_i}+L_i \D(\bm x+U_i(y_i-x_i),\bm x);
    \end{equation}
    \item  \label{fac:relSmoothEqvi3}
    for all \((\bm x,\bm y)\in\interior\dom h\times \interior\dom h\) and $i=1,\ldots,N$,
    \begin{equation}\label{eq:upperEqvi}
        \innprod{\nabla_i f(\bm x)-\nabla_i f(\bm y)}{x_i-y_i}\leq L_i \innprod{\nabla_i h(\bm x)-\nabla_i h(\bm y)}{x_i-y_i};
    \end{equation}
    \item \label{fac:relSmoothEqvi4}
    if \(f\in\mathcal{C}^2(\interior\dom f)\) and 
    \(h\in\mathcal{C}^2(\mathrm{\bf int}\dom h)\)
    for all \(\bm x\in\interior\dom h\), then 
    \begin{equation}\label{eq:descentEqviC2}
        L_i \nabla_{x_ix_i}^2 h(\bm x)-\nabla_{x_ix_i}^2 f(\bm x)\succeq 0,
    \end{equation}
    for $i=1,\ldots,N$.
    \end{enumerateq}
\end{prop}

\begin{proof}
Fixing all the blocks except one of them, the results can be concluded in the same way as \cite[Proposition 1.1]{lu2018relatively}.
\end{proof}

\subsection{Bregman proximal alternating linearized mapping}\label{sec:bregProx}
Recall that if $N=1$, for a kernel function \(\func{h}{\R^n}{\Rinf}\) and a proper lower semicontinuous function \(\func{g}{\R^n}{\Rinf}\), the Bregman proximal mapping is given by
\begin{equation}\label{eq:bprox}
	\prox_{\gamma g}^h(x):=
	\argmin_{z\in\R^n}\set{
		g(z) + \tfrac{1}{\gamma}\D_h(z,x)
	}.
\end{equation}
which is a generalization of the classical one using the Bregman distance \eqref{eq:bregman} in place of the Euclidean distance; see, e.g., \cite{chen1993convergence} and references therein. We note that
\[
\prox_{\gamma g}^h(x)=\set{y\in\dom g\cap\dom h ~|~ g(y)+\tfrac{1}{\gamma} \D_h(y,x)=\min_z \set{g(z) +\tfrac{1}{\gamma}\D_h(z,x)	}<+\infty},
\]
which implies $\dom \prox_{\gamma g}^h \subset \mathbf{int}\dom h,~\range \prox_{\gamma g}^h \subset \dom g\cap\dom h$. The function $g$ is \DEF{$h$-prox-bounded} if there exists $\gamma>0$
such that $\min_z \set{g(z) +\tfrac{1}{\gamma}\D_h(z,x)}>-\infty$ for some $x\in\R^n$; cf. \cite{ahookhosh2019bregman}. We next extend this definition to our multi-block setting.

\begin{defin}[multi-block $h$-prox-boundedness] 
\label{def:bregman} 
    A function $\func{g}{\R^{n}}{\Rinf}$ is \DEF{multi-block $h$-prox-bounded} if for each $i\in \set{1,\ldots,N}$ there exists $\gamma_i>0$ and $\bm x\in\R^n$ such that 
    \[
        g^{\nicefrac{h}{\gamma_i}}(\bm x):=\min_{z\in\R^{n_i}} \set{g(\bm x+U_i(z-x_i)) +\tfrac{1}{\gamma_i}\D(\bm x+U_i(z-x_i),\bm x)}>-\infty.
    \]
The supremum of the set of all such $\gamma_i$ is the threshold $\gamma_{i,g}^h$ of the $h$-prox-boundedness, i.e.,
\begin{equation}\label{eq:dProxBound}
    \gamma_{i,g}^{h}:=\sup \set{\gamma_i>0 ~|~ \exists \bm x\in\R^n~\mathrm{s.t.}~
    g^{\nicefrac{h}{\gamma_i}}(\bm x)>-\infty}.
\end{equation}
\end{defin}

For the problem \eqref{eq:P}, we have $g=\sum_{i=1}^N g_i$ leading to
\begin{equation}\label{eq:ghGamma}
	g^{\nicefrac{h}{\gamma_i}}(\bm x)=\sum_{j\neq i} g_i(x_i)+\min_{z\in\R^{n_i}} \set{g_i(z) +\tfrac{1}{\gamma_i}\D(\bm x+U_i(z-x_i),\bm x)},
\end{equation}
i.e., we therefore denote $\gamma_{g_i}^h:=\gamma_{i,g}^h$. 
If $g$ is multi-block $h$-prox-bounded for $\overline{\gamma}_i>0$, so is for all $\gamma_i\in{(0,\overline{\gamma}_i)}$. We next present equivalent conditions to this notion.

\begin{prop}[characteristics of multi-block $h$-prox-boundedness] 
\label{pro:proxBoundedness}%
	For a multi-block kernel function \(\func{h}{\R^n}{\Rinf}\) and proper and lsc functions $\func{g_i}{\R^{n_i}}{\Rinf}$ with $i=1,\ldots,N$, the following statements are equivalent:
	\begin{enumerateq}
	\item\label{pro:proxBoundedness1}
		$g=\sum_{i=1}^N g_i$ is multi-block \(h\)-prox-bounded;
	\item\label{pro:proxBoundedness2}
		for all $i=1,\ldots,N$ and $h_{\bm x}^i$ given in \eqref{eq:hxi}, $g_i+r_i h_{\bm x}^i$ is bounded below on $\R^{n_i}$ for some $r_i\in\R$;
	\item\label{pro:proxBoundedness3}
		for all $i=1,\ldots,N$, $\liminf_{\|z\|\to \infty}\nicefrac{g_i(z)}{h_{\bm x}^i(z)}>-\infty$.
	\end{enumerateq}
\end{prop}

\begin{proof}
		Suppose $g^{\nicefrac{h}{\gamma_i}}(\bm x)>-\infty$ and let $r_i>\tfrac{1}{\gamma_i}$.
		Then, for all $i=1,\ldots,N$, it holds that
		\begin{align*}
			g_i(z)&+r_i h_{\bm x}^i(z)
		=
		    g_i(z)+ \tfrac{1}{\gamma}\D(\bm x+U_i(z-x_i),\bm x)+r_i h_{\bm x}^i(z)- \tfrac{1}{\gamma}\D(\bm x+U_i(z-x_i),\bm x)
		\\
		&\geq
			g^{\nicefrac{h}{\gamma_i}}(\bm x)-\sum_{j\neq i} g_i(x_i)
			{}+{}
			\tfrac{r_i\gamma_i-1}{\gamma_i}h_{\bm x}^i(z)
			{}+{}
			\tfrac{1}{\gamma_i}(h(\bm x)+\innprod{\nabla_i h(\bm x)}{z-x_i})
		{}\eqqcolon{}
			\tilde g_i(z).
		\end{align*}
		Notice that \(\tilde g_i\) is strictly convex and coercive, and as such is lower bounded.
		Conversely, suppose that $\alpha_i\coloneqq\inf g_i+r_i h_{\bm x}^i>-\infty$.
		Then, from \eqref{eq:ghGamma}, we obtain
		\begin{align*}
			g^{\nicefrac{h}{\gamma_i}}(\bm x)
			&=\sum_{j\neq i} g_i(x_i)+\min_{z\in\R^{n_i}} \set{g_i(z) +\tfrac{1}{\gamma_i}\D(\bm x+U_i(z-x_i),\bm x)},\\
			&\geq
				\sum_{j\neq i} g_i(x_i)+\alpha_i + \inf_z\set{-r_i h_{\bm x}^i(z) + \tfrac{1}{\gamma_i}\D(\bm x+U_i(z-x_i),\bm x)}			\\
			&\geq
				\sum_{j\neq i} g_i(x_i)+\alpha_i - \tfrac{1}{\gamma_i} h(\bm x)+\tfrac{1}{\gamma_i} \innprod{\nabla_i h(\bm x)}{x_i}+ \inf_z\set{\tfrac{1-\gamma_i r_i}{\gamma_i} h_{\bm x}^i(z) - \tfrac{1}{\gamma_i} \innprod{\nabla_i h(\bm x)}{z}},
		\end{align*}
		which is finite, owing to \(1\)-coercivity of $z\mapsto \tfrac{1-\gamma_i r_i}{\gamma_i} 
		h_{\bm x}^i(z) - \tfrac{1}{\gamma_i} \innprod{\nabla_i h(\bm x)}{z}$.
		
		Suppose that $\alpha_i\coloneqq\inf g_i+r_i h_{\bm x}^i>-\infty$.
		Since $h_{\bm x}^i(\cdot)$ is $1$-coercive, we have 
		\[
			\liminf_{\|z\|\to\infty} \tfrac{g_i(z)}{h_{\bm x}^i(z)} \geq -r_i+ \liminf_{\|z\|\to\infty} \tfrac{\alpha_i}{h_{\bm x}^i(z)}=-r_i>-\infty.
		\]
		Conversely, suppose $\liminf_{\|z\|\to \infty} \nicefrac{g_i(z)}{h_{\bm x}^i(z)}>-\infty$.
		Then, there exists \(\ell_i,M_i\in\R\) such that \(\nicefrac{g_i(z)}{h_{\bm x}^i(z)}\geq\ell_i\) whenever \(\|z\|\geq M_i\).
		In particular
		\[
			\inf_{\|z\|\geq M_i}g_i(z)+r_i h_{\bm x}^i(z)
		{}\geq{}
			\inf_{\|x\|\geq M_i}h_{\bm x}^i(z)(\ell_i+r_i)
		{}>{}
			-\infty,
		\]
		where the last inequality follows from coercivity of \(h_{\bm x}^i\).
		Since \(\inf_{\|z\|\leq M_i}{g_i(z)+r_i h_{\bm x}^i(z)}>-\infty\) owing to lower semicontinuity, we conclude that \(g_i+r_i h_{\bm x}^i\) is lower bounded on \(\R^n\).
		\qedhere
\end{proof}

Let us now define the function $\func{\M}{\R^n\times\R^n}{\Rinf}$ as
\begin{equation}\label{eq:modelM}
   \M(\bm z,\bm x):= \innprod{\nabla f(\bm x)}{\bm z-\bm x}+\tfrac{1}{\gamma}\D(\bm z,\bm x)+\sum_{i=1}^N g_i(z_i)
\end{equation}
and the set-valued \DEF{Bregman proximal alternating linearized mapping} $\ffunc{\T}{\R^n}{\R^{n_i}}$ as
\begin{equation}\label{eq:tx}
    \T(\bm x):= \argmin_{z\in\R^{n_i}} \M[][\nicefrac{h}{\gamma_i}](\bm x+U_i(z-x_i),\bm x),
\end{equation}
which reduces to the Bregman forward-backward splitting mapping if $N=1$; cf. \cite{bolte2018first,ahookhosh2019bregman}. 

\begin{rem}[majorization model]
    Note that invoking \Cref{fac:relSmoothEqvi2}, the multi-block ($L_1,\ldots,L_N$)-relative smoothness assumption of $f$ entails a \DEF{majorization model} 
    \begin{align*}
    	\varphi(\bm x+U_i(y_i-x_i))&\leq f(\bm x)+\innprod{\nabla_i f(\bm x)}{y_i-x_i}+L_i \D(\bm x+U_i(y_i-x_i),\bm x)+g_i(y_i)+\sum_{j\neq i} g_j(x_j)\\
    	&\leq f(\bm x)+\innprod{\nabla_i f(\bm x)}{y_i-x_i}+\tfrac{1}{\gamma_i} \D(\bm x+U_i(y_i-x_i),\bm x)+g_i(y_i)+\sum_{j\neq i} g_j(x_j),
    \end{align*}
    for $\gamma_i\in(0,\nicefrac{1}{L_i})$. 
\end{rem}

In the next lemma, we show that the cost function $\varphi$ is monotonically decreasing by minimizing the model \eqref{eq:modelM} with respect to each block of variables.

\begin{lem}[Bregman proximal alternating inequality]
\label{lem:proxAltIneq}
    Let the conditions in \Cref{ass:basic:fgh} hold, and let $\overline{z}\in\T(\bm x)$ with $\gamma_i\in(0,\nicefrac{1}{L_i})$. Then, 
    \begin{equation}\label{eq:fiOverx}
        \varphi(\bm x+U_i(\overline{z}-x_i))\leq\varphi(\bm x)-\tfrac{1-\gamma_i L_i}{\gamma_i}\D(\bm x+U_i(\overline{z}-x_i),\bm x),
    \end{equation}
    for all $i=1,\ldots,N$.
\end{lem}

\begin{proof}
    For $i\in\set{1,\ldots,N}$, \eqref{eq:tx} is simplified in the form
    \begin{equation}\label{eq:tix}
        \begin{split}
            \T(\bm x)&= \argmin_{z\in\R^{n_i}} \set{\innprod{\nabla f(\bm x)}{U_i(z-x_i)} +\tfrac{1}{\gamma_i}\D(\bm x+U_i(z-x_i),\bm x)+\sum_{i=1}^N g_i(z)}\\
            &= \argmin_{z\in\R^{n_i}} \set{\innprod{\nabla_i f(\bm x)}{z-x_i} +\tfrac{1}{\gamma_i}\D(\bm x+U_i(z-x_i),\bm x)+ g_i(z)}.
        \end{split}
    \end{equation}
    Considering $\overline{z}\in\T(\bm x)$, we have
    \[
        \innprod{\nabla_i f(\bm x)}{\overline{z}-x_i} +\tfrac{1}{\gamma_i}\D(\bm x+U_i(\overline{z}-x_i),\bm x) +g_i(\overline{z})\leq g_i(x_i).
    \]
    Since $f$ is $(L_1,\ldots,L_N)$-smooth relative to $h$, it follows from \Cref{fac:relSmoothEqvi2} for $\bm x$ and $y_i=\overline{z}$ that
    \begin{align*}
        f(\bm x&+U_i(\overline{z}-x_i)) \leq f(\bm x)+\innprod{\nabla_i f(\bm x)}{\overline{z}-x_i}+L_i \D(\bm x+U_i(\overline{z}-x_i),\bm x)\\
        &\leq f(\bm x)+L_i \D(\bm x+U_i(\overline{z}-x_i),\bm x) +g_i(x_i)-g_i(\overline{z})-\tfrac{1}{\gamma_i} L_i \D(\bm x+U_i(\overline{z}-x_i),\bm x)\\
        &= f(\bm x)+g_i(x_i)-g_i(\overline{z})-\tfrac{1-\gamma_i L_i}{\gamma_i}  \D(\bm x +U_i(\overline{z}-x_i),\bm x),
    \end{align*}
    giving \eqref{eq:fiOverx}. 
\end{proof}

Recall that a function $\func{\vartheta}{\R^n\times\R^m}{\Rinf}$ with values $\vartheta(x,u)$ is \DEF{level-bounded} in $x$ locally uniformly in $u$ if for each $\bar u\in\R^m$ and $\alpha\in\R$ there is a neighborhood $\mathcal U$ of $\bar u$ along with a bounded set $B\subset\R^n$ such that $[\vartheta\leq \alpha]\subset B$ for all $u\in\mathcal U$, cf. \cite{rockafellar2011variational}. Using this definition, the fundamental properties of the mapping $\T$ are investigated in the subsequent result.

\begin{prop}[properties of Bregman proximal alternating linearized mapping]
\label{pro:proxPro} 
    Under conditions given in \Cref{ass:basic:fgh} for $i=1,\ldots,N$, the following statements are true:
    \begin{enumerate}
    \item \label{pro:proxPro3} 
    $\T(\bm x)$ is nonempty, compact, and outer semicontinuous (osc) for all $\bm x\in\interior\dom h$;
    \item \label{pro:proxPro1} 
    $\dom \T=\interior\dom h$; 
    \item \label{pro:proxPro2}
    If $\overline{z}\in\T(\bm x)$, then $\bm x+U_i(\overline{z}-x_i)\subseteq\dom \varphi \subseteq \interior\dom h$;
    \end{enumerate}\let\qedsymbol\relax
\end{prop}

\begin{proof}
    For a fixed $\gamma_i^0\in(0,\gamma_{g_i}^h)$ and a vector $\bm x\in\interior\dom h$, let us define the function $\func{\Phi_i}{\R^{n_i}\times\R^{n}\times\R}{\Rinf}$ given by 
    \begin{equation*}
    \Phi_i(z,\bm x,\gamma_i):=g_i(z)+\innprod{\nabla_i f(\bm x)}{z-x_i}+
    \left\{
    \begin{array}{ll}
        \tfrac{1}{\gamma_i} \D(\bm x+U_i(z-x_i),\bm x)  &\quad \mathrm{if}~\gamma_i\in (0,\gamma_i^0],\vspace{1mm} \\
        0 & \quad \mathrm{if}~\gamma_i=0~\mathrm{and}~z=x_i,\vspace{1mm}\\
        +\infty & \quad \mathrm{otherwise}.
    \end{array}
    \right.
    \end{equation*} 
    Since $f$ and $g_i$ are proper and lsc, so is $\Phi_i$ on the set 
    $\set{(z,\bm x,\gamma_i)}[\|z-x_i\|\leq \mu \gamma_i,~ 0\leq\gamma_i\leq \gamma_i^0]$,
    for a constant $\mu>0$. We show that $\Phi_i$ is level-bounded in $z$ locally
    uniformly in $(\bm x,\gamma)$. If it is not, then there exists $\seq{\bm x^k}\subset\interior\dom h$, $\seq{z^k}$ with $\bm x^k+U_i(z^k-x_i^k)\subset\interior\dom h$, and $\seq{\gamma_i^k}\subset (0,\gamma_i^0]$
    such that $\Phi_i(z^k,\bm x^k,\gamma_i^k)\leq \beta<\infty$ with $(\bm x^k,\gamma^k)\to (\overline{\bm x},\overline{\gamma}_i)$ and
    $\|z^k\|\to\infty$. This guarantees that, for sufficiently large $k$, $z^k\neq x_i^k$, i.e., 
    $\gamma_i^k\in (0,\gamma_0]$ and 
    \begin{equation*}\label{eq:gzkBeta}
       g_i(z^k)+\innprod{\nabla_i f(\bm x^k)}{z^k-x_i^k}+\tfrac{1}{\gamma_i^k}\D_h(\bm x^k+U_i(z^k-x_i^k),\bm x^k)\leq\beta. 
    \end{equation*}
    Setting $\tilde \gamma_i\in (\gamma_i^0,\gamma_i^h)$, \Cref{pro:proxBoundedness2} ensures that there exists a constant $\tilde{\beta}\in\R$ such that  
    \[
    		g_i(z^k)+\tfrac{1}{\tilde{\gamma}_i} h(\bm x^k+U_i(z^k-x_i^k))\geq g_i(z^k)+r_i h(\bm x^k+U_i(z^k-x_i^k))\geq \tilde{\beta}
    \]
	Subtracting the last two inequalities, it holds that
    \begin{align*}
     \innprod{\nabla_i f(\bm x^k)}{z^k-x_i^k}
    +\tfrac{1}{\gamma_i^k}\D(\bm x^k+U_i(z^k-x_i^k),\bm x^k)
    -\tfrac{1}{\tilde{\gamma}_i} h(\bm x^k+U_i(z^k-x_i^k))
    \leq\beta-\tilde{\beta}.   
    \end{align*}
    Expanding $\D(\bm x^k+U_i(z^k-x_i^k),\bm x^k)$, dividing both sides by $\|z^k\|$, and taking
    limit from both sides of this inequality as $k\to\infty$, it can be deduced that
    \begin{align*}
     \lim_{k\to\infty} \left(
     \Innprod{\nabla_i f(\bm x^k)-\tfrac{1}{\gamma_i^k}\nabla_i h(\bm x^k)}{\frac{z^k-x_i^k}{\|z^k\|}}-\tfrac{1}{\gamma_i^k}\frac{h(\bm x^k)}{\|z^k\|}\right)
    +\left(\tfrac{1}{\gamma_i^k}-\tfrac{1}{\tilde{\gamma}_i}\right)\lim_{k\to\infty}\frac{h_{\bm x^k}^i(z^k)}{\|z^k\|} \leq \lim_{k\to\infty}\frac{\beta-\tilde{\beta}}{\|z^k\|}. 
    \end{align*}
    This leads to the contradiction $+\infty\leq 0$, which implies that $\Phi_i$ is
    level-bounded. Therefore, all assumptions of the parametric minimization theorem  
    \cite[Theorem~2.2 and Corollary~2.2]{kan2012moreau} are satisfied, i.e., \Cref{pro:proxPro1}. If $\overline{z}\in\T(\bm x)$, then \Cref{lem:proxAltIneq} implies that $\varphi(\bm x+U_i(\overline{z}-x_i))\leq \varphi(\bm x)<\infty$ for $i=1,\ldots,N$, i.e., 
    $\bm x+U_i(\overline{z}-x_i) \subseteq\dom \varphi$, the second inclusion follows from \Cref{ass:basic:argmin}.
\end{proof}
\begin{rem}[sum or product separable kernel] 
\label{rem:addProdSepKernel}
	Let us observe the following.
    \begin{enumerate}
        \item \label{rem:addProdSepKernel1}
            If $h$ is an \DEF{additive separable} function, i.e., $h(x_1,\ldots,x_N)=h_1(x_1)+\ldots+h_N(x_N)$, then \eqref{eq:tx} can be written in the form
            \begin{align*}
                \T(\bm x)&=\argmin_{z\in\R^{n_i}} \set{g_i(z)+\innprod{\nabla_i f(\bm x)}{z-x_i} +\tfrac{1}{\gamma_i}\D(\bm x+U_i(z-x_i),\bm x)}\\
                &=\argmin_{z\in\R^{n_i}} \set{g_i(z) +\tfrac{1}{\gamma_i}(h_i(z)-h_i(x_i)-\innprod{\nabla h_i(x_i)-\gamma_i \nabla_i f(\bm x)}{z-x_i})}\\
                &=\argmin_{z\in\R^{n_i}} \set{g_i(z) +\tfrac{1}{\gamma_i} \D[h_i](z,\nabla h_i^\star(\nabla h_i(x_i)-\gamma_i \nabla_i f(\bm x))}\\
                &=\prox_{\gamma_i g_i}^{h_i}(\nabla h_i^\star(\nabla h_i(x_i)-\gamma_i \nabla_i f(\bm x))).
            \end{align*}
        \item \label{rem:addProdSepKernel2}
            If $h$ is \DEF{product separable}, i.e., $h(x_1,\ldots,x_N)=h_1(x_1)\times\ldots\times h_N(x_N)$, then
            \begin{align*}
                \T(\bm x)&=\argmin_{z\in\R^{n_i}} \set{g_i(z)+\innprod{\nabla_i f(\bm x)}{z-x_i} +\tfrac{1}{\gamma_i}\D(\bm x+U_i(z-x_i),\bm x)}\\
                &=\argmin_{z\in\R^{n_i}} \set{g_i(z) +\tfrac{\eta_{\bm x}^i}{\gamma_i}(h_i(z)-h_i(x_i)-\innprod{\nabla h_i(x_i)-\tfrac{\gamma_i}{\eta_{\bm x}^i} \nabla_i f(\bm x)}{z-x_i})}\\
                &=\argmin_{z\in\R^{n_i}} \set{g_i(z) +\tfrac{1}{\mu_i} \D[h_i](z,\nabla h_i^\star(\nabla h_i(x_i)-\mu_i \nabla_i f(\bm x))}\\
                &=\prox_{\mu_i g_i}^{h_i}(\nabla h_i^\star(\nabla h_i(x_i)-\mu_i \nabla_i f(\bm x))),
            \end{align*}
            where $\mu_i:=\nicefrac{\gamma_i}{\eta_{\bm x}^i}$ and $\eta_{\bm x}^i:= \prod_{j\neq i} h_j(x_j)\neq 0$.
    \end{enumerate}
\end{rem}

			
	\section{Multi-block Bregman proximal alternating linearized minimization} \label{sec:convAnalysis}
		
We here introduce a multi-block proximal alternating linearized minimization algorithm and investigate its subsequential and global convergence, along with its convergence rate.

For a given point $\bm x^k=(x_1^k,\ldots,x_N^k)$, we set
\begin{align*}
    \bm x^{k,i}:=(x_1^{k+1},\ldots,x_i^{k+1},x_{i+1}^k,\ldots,x_N^k),
\end{align*}
i.e., $\bm x^{k,0}=\bm x^k$ and $\bm x^{k,N}=\bm x^{k+1}$.
Using this notation and \eqref{eq:tx}, we next introduce the multi-block Bregman proximal alternating linearized minimization (\refBPALM[]) algorithm.

\begin{algorithm*}[ht] 
\algcaption{({\bf BPALM}) Bregman Proximal Alternating Linearized Minimization}%
\begin{algorithmic}[1]
\Require{%
	\(\gamma_i\in(0,\nicefrac{1}{L_i}),~ i=1,\ldots,N\),~
	\(\bm x^0\in\R^{n_1}\times\ldots\times\R^{n_N}\),\ $I_n=(U_1,\ldots,U_N)\in\R^{n\times n}$ with $U_i\in\R^{n\times n_i}$ and the identity matrix $I_n$.%
}%
\Initialize{$k=0$.}%
\While{some stopping criterion is not met}
\State\label{state:bpalmm0Xk1}%
    $\bm x^{k,0}=\bm x^k$;
    \For{\texttt{$i=1,\ldots,N$}} compute
        \begin{align}\label{eq:xik1}
            x_i^{k,i}\in \T(\bm x^{k,i-1}),\quad \bm x^{k,i}=\bm x^{k,i-1}+U_i(x_i^{k,i}-x_i^{k,i-1});
        \end{align}
    \EndFor
\State\label{state:bpalmk1}%
	 $\bm x^{k+1}=\bm x^{k,N}$,~$k= k+1$;%
\EndWhile
\Ensure{%
	A vector $\bm x^k$.%
}%
\end{algorithmic}
\label{alg:bpalm}
\end{algorithm*}%

We note that each iteration of \refBPALM[] requires one call of the first-order oracle for the information needed in \eqref{eq:xik1}, and the iteration \eqref{eq:xik1} are well-defined by \Cref{pro:proxPro}. In addition, notice that if $N=1$, this algorithm reduces to the common (Bregman) proximal gradient (forward-backward) method  \cite{bauschke2016descent,beck2009fast,bolte2018first}; if $N=2$, $h,f\in\mathcal{C}^2$, and $g_1=g_2= 0$, then it reduces to B-PALM \cite{li2019provable}; if $N=2$ and $h(\bm x)=\tfrac{1}{2}(\|x_1\|^2+\|x_2\|^2)$, it reduces to PALM \cite{bolte2014proximal}; if $h(\bm x)=\tfrac{1}{2}\sum_{i=1}^N \|x_i\|^2$, then this algorithm is reduced to C-PALM \cite{shefi2016rate}. 

We begin with showing some basic properties of the sequence generated by \refBPALM[], involving a \DEF{sufficient decrease condition}.

\begin{prop}[sufficient decrease condition]
\label{pro:suffDecreaseIneq}
    Let \Cref{ass:basic:fgh} hold, and let $\seq{\bm x^k}$ be generated by \refBPALM[]. Then, the following statements are true:
    \begin{enumerate}
        \item \label{pro:suffDecreaseIneq1}
            the sequence $\seq{\varphi(\bm x^k)}$ is nonincreasing and
            \begin{equation}\label{eq:dh1dh2Ineq2}
                \rho \sum_{i=1}^N \D(\bm x^{k,i},\bm x^{k,i-1})\leq \varphi(\bm x^k)-\varphi(\bm x^{k+1}),
            \end{equation}
            where $\rho:=\min\set{\tfrac{1-\gamma_1 L_1}{\gamma_1},\ldots,\tfrac{1-\gamma_N L_N}{\gamma_N}}$;
        \item \label{pro:suffDecreaseIneq2}
            we have
            \begin{equation}\label{eq:dh1dh2Ineq3}
                \sum_{k=1}^\infty \sum_{i=1}^N\D(\bm x^{k,i},\bm x^{k,i-1}) < \infty,
            \end{equation}
            i.e., $\lim_{k\to\infty} \D(\bm x^{k,i},\bm x^{k,i-1})=0$ for $i=1,\ldots,N$.
    \end{enumerate}
\end{prop}

\begin{proof}
    Plugging  $\overline{z}=x_i^{k,i}$ and $\bm x=\bm x^{k,i-1}$ into \Cref{lem:proxAltIneq}, it holds that
    \begin{equation}\label{eq:fiOverx1}
        \varphi(\bm x^{k,i})\leq\varphi(\bm x^{k,i-1})-\tfrac{1-\gamma_i L_i}{\gamma_i}\D(\bm x^{k,i},\bm x^{k,i-1}).
    \end{equation}
    Summing up both sides of \eqref{eq:fiOverx1} from $i=1$ to $N$, it follows that
    \begin{align*}
        \sum_{i=1}^N \tfrac{1-\gamma_i L_i}{\gamma_i} \D(\bm x^{k,i},\bm x^{k,i-1})&\leq \sum_{i=1}^N [\varphi(\bm x^{k,i-1})-\varphi(\bm x^{k,i})]
        =\varphi(\bm x^k)-\varphi(\bm x^{k+1}),
    \end{align*}
    giving \eqref{eq:dh1dh2Ineq2}. Let us sum up both sides  of \eqref{eq:dh1dh2Ineq2} from $k=0$ to $q$:
    \begin{align*}
        \rho\sum_{k=0}^{q} \sum_{i=1}^N  \D(\bm x^{k,i},\bm x^{k,i-1}) &\leq 
        \sum_{k=0}^{q} \varphi(\bm x^k)-\varphi(\bm x^{k+1})=\varphi(\bm x^{0})-\varphi(\bm x^{q+1})
        \leq\varphi(\bm x^0)-\underline{\varphi}
        < \infty.
    \end{align*}
    Taking the limit as $q\to+\infty$, \eqref{eq:dh1dh2Ineq3} holds true. Together with  $\D_h(\cdot,\cdot)\geq 0$, this  proves the claim.
\end{proof}

Let us consider the condition
\begin{equation}\label{eq:stopCrit}
    \sum_{i=1}^N \D(\bm x^{k,i},\bm x^{k,i-1})\leq \varepsilon
\end{equation}
as a \DEF{stopping criterion}, for the accuracy parameter $\varepsilon>0$. Then, the first main consequence of \Cref{pro:suffDecreaseIneq} will provide us the \DEF{iteration complexity} of \refBPALM[], which is the number of iterations needed for the stopping criterion \eqref{eq:stopCrit} to be satisfied. 

\begin{cor}[iteration complexity]
\label{cor:iterComplexity}
    Let \Cref{ass:basic:fgh} hold, and let $\seq{\bm x^k}$ be generated by \refBPALM[], and let \eqref{eq:stopCrit} be the stopping criterion. Then, \refBPALM[] will be terminated within $k\leq1+ \tfrac{\varphi(x^0)-\inf \varphi}{\rho \varepsilon}$ iterations.
\end{cor}

\begin{proof}
    Summing both sides of \eqref{eq:dh1dh2Ineq2} over the first $\mathcal{K}\in\N$ iterations and telescoping the right hand side, it holds that
    \[
        \rho \sum_{k=0}^{\mathcal{K}-1} \sum_{i=1}^N \D(\bm x^{k,i},\bm x^{k,i-1})\leq \sum_{k=0}^{\mathcal{K}-1} \left(\varphi(\bm x^k)-\varphi(\bm x^{k+1})\right)=\varphi(\bm x^0)-\varphi(\bm x^\mathcal{K}) \leq \varphi(\bm x^0)-\inf \varphi.
    \]
    Assuming that for all $(\mathcal{K}-1)$-th iterations the stopping criterion \eqref{eq:stopCrit} is not satisfied, i.e., $\sum_{i=1}^N \D(\bm x^{k,i},\bm x^{k,i-1})>\varepsilon$, which leads to $\mathcal{K}\leq 1+\tfrac{\varphi(x^0)-\inf \varphi}{\rho \varepsilon}$, giving the desired  result.
\end{proof}

In order to show the subsequential convergence of the sequence $\seq{\bm x^k}$ generated by \refBPALM[], the next proposition will provide a lower bound for iterations gap $\|\bm x^{k+1}-\bm x^k\|$ using the subdifferential of $\partial \varphi(\bm x^{k+1})$.

\begin{prop}[subgradient lower bound for iterations gap]
\label{pro:subgradLowBound1}
    Let \Cref{ass:basic:fgh} hold, and let $\seq{\bm x^k}$ be generated by \refBPALM[] that we assume to be bounded. For a fixed $k\in\N$, we define
    \begin{equation}\label{eq:pxkp1}
        \mathcal{G}_i^{k+1}:= \tfrac{1}{\gamma_i} (\nabla_i h(\bm x^{k,i-1})-\nabla_i h(\bm x^{k,i})) +\nabla_i f(\bm x^{k,i})-\nabla_i f(\bm x^{k,i-1})\quad  i=1,\ldots,N.
    \end{equation}
    Then, $\left(\mathcal{G}_1^{k+1},\ldots,\mathcal{G}_N^{k+1}\right)\in \partial\varphi(\bm x^{k+1})$ and
    \begin{equation}\label{eq:subGradUpBound}
        \|(\mathcal{G}_1^{k+1},\ldots,\mathcal{G}_N^{k+1})\|\leq \overline c \sum_{i=1}^N \|x_i^{k+1}-x_i^k\|,
    \end{equation}
    with $\overline c:=\max \set{\tfrac{\widetilde L+ \gamma_1\widehat L}{\gamma_1},\ldots,\tfrac{\widetilde L+ \gamma_{N-1}\widehat L}{\gamma{N-1}},\tfrac{\widetilde L_{N}+ \gamma_N \widehat L_{N}}{\gamma_N}}$ in which $\widehat L$ and $\widetilde L>0$ are Lipschitz moduli of $\nabla_i f$, $\nabla_i h$ ($i=1,\ldots,N-1$) on bounded sets, and $\widehat L_N$ and $\widetilde L_N>0$ are Lipschitz moduli of $\nabla_N f$, $\nabla_N h$ on bounded sets, respectively.
\end{prop}

\begin{proof}
    The optimality conditions for \eqref{eq:xik1} ensures that there exists $q_i^{k,i}\in\partial g_i(x_i^{k,i})$ such that
    \[
        \nabla_i f(\bm x^{k,i-1})+\tfrac{1}{\gamma_i} \left(\nabla_i h(\bm x^{k,i})-\nabla_i h(\bm x^{k,i-1})\right)+q^{k,i}=0 \quad i=1,\ldots,N,
    \]
    leading to
    \begin{equation}\label{eq:gradEq11}
        q^{k,i}=\tfrac{1}{\gamma_i} \left(\nabla_i h(\bm x^{k,i-1})-\nabla_i h(\bm x^{k,i}))\right)-\nabla_i f(\bm x^{k,i-1})\quad i=1,\ldots,N.
    \end{equation}
    On the other hand, owing to \cite[Proposition 2.1]{attouch2010proximal}, the subdifferential of $\varphi$ is given by
    \[
        \partial \varphi(\bm x)=(\partial_{1} \varphi(\bm x),\ldots,\partial_{N} \varphi(\bm x))
        =(\nabla_1 f(\bm x)+\partial g_1(x_1),\ldots,\nabla_N f(\bm x)+\partial g_N(x_N)),
    \]
    i.e., for $\bm x=\bm x^{k+1}$,
    \[
        \nabla_i f(\bm x^{k+1})+\partial g_i(x_i^{k+1})\in \partial_{i} \varphi(\bm x^{k+1})\quad i=1,\ldots,N,
    \]
    which means $(\mathcal{G}_1^{k+1},\ldots,\mathcal{G}_N^{k+1})\in \partial \varphi(\bm x^{k+1})$. It follows from the Lipschitz continuity of $\nabla_i f$, $\nabla_i h$ on bounded sets and the assumption of $\seq{\bm x^k}$ being bounded that there exist $\widehat L$, $\widehat L_N$, $\widetilde L>0$ and $\widetilde L_N>0$ such that
    \begin{align*}
        \|\mathcal{G}_i^{k+1}\| \leq \tfrac{1}{\gamma_i} \|\nabla_i h(\bm x^{k,i-1})-\nabla_i h(\bm x^{k,i})\| +\|\nabla_i f(\bm x^{k,i})-\nabla_i f(\bm x^{k,i-1})\|\leq \tfrac{\widetilde L+ \gamma_i\widehat L}{\gamma_i} \sum_{i=1}^{N}\|x_i^{k+1}-x_i^k)\|,
    \end{align*}
    for $i=1,\ldots,N-1$, and
    \begin{align*}
        \|\mathcal{G}_N^{k+1}\| \leq \tfrac{1}{\gamma_N} \|\nabla_N h(\bm x^{k,N-1})-\nabla_N h(\bm x^{k,N})\| +\|\nabla_N f(\bm x^{k+1})-\nabla_{N} f(\bm x^{k,N-1})\|
        \leq \tfrac{\widetilde L_N+ \gamma_N \widehat L_N}{\gamma_N} \|x_N^k-x_N^{k+1}\|.
    \end{align*}
    Invoking the last two inequalities, it can be concluded that
    \begin{align*}
        \|(\mathcal{G}_1^{k+1},\ldots,\mathcal{G}_N^{k+1})\|
        \leq \max \set{\tfrac{\widetilde L+ \gamma_1\widehat L}{\gamma_1},\ldots,\tfrac{\widetilde L+ \gamma_{N-1}\widehat L}{\gamma{N-1}},\tfrac{\widetilde L_{N}+ \gamma_N \widehat L_{N}}{\gamma_N}} \sum_{i=1}^{N}\|x_i^{k+1}-x_i^k)\|,
    \end{align*}
    as claimed.
\end{proof}

Next, we proceed to derive the \DEF{subsequential convergence} of the sequence $\seq{\bm x^k}$ generated by \refBPALM[]: every cluster point of $\seq{\bm x^k}$ is a critical point of $\varphi$. Further, we explain some fundamental properties of the set of all cluster points $\omega(\bm x^0)$ of the sequence $\seq{\bm x^k}$. 

\begin{thm}[subsequential convergence and properties of $\omega(\bm x^0)$]
\label{pro:chatClusterPoint10}
    Let \Cref{ass:basic:fgh} hold, let the kernel $h$ be locally multi-block strongly convex, and let $\seq{\bm x^k}$ be generated by \refBPALM[] that we assume to be bounded. Then the following statements are true:
    \begin{enumerate}
        \item \label{pro:chatClusterPoint11}
            $\emptyset\neq \omega(\bm x^0)\subset \mathbf{crit}~ \varphi$;
        \item \label{pro:chatClusterPoint12}
            $\lim_{k\to\infty}\dist\left(\bm x^k,\omega(\bm x^0)\right)=0$;
        \item \label{pro:chatClusterPoint13}
        $\omega(\bm x^0)$ is a nonempty, compact, and connected set;
        \item \label{pro:chatClusterPoint14}
        the objective function $\varphi$ is finite and constant on $\omega(\bm x^0)$.
    \end{enumerate}
\end{thm}

\begin{proof}
    For a limit point $\bm x^\star=(x_1^\star,\ldots,x_N^\star)$ of the sequence $\seq{\bm x^k}$, it follows from the boundedness of this sequence that there exists an infinite index set $\mathcal{J}\subset \N$ such that the subsequence $\seq{\bm x^k}[k\in\mathcal{I}]$ converges to $\bm x^\star$ as $k\to\infty$. From the lower semicontinuity of $g_i$ ($i=1,\ldots,N$) and for $j\in\mathcal{J}$, it can be deduced that
    \begin{equation}\label{eq:lsc0}
        \liminf_{j\to\infty} g_i(x_i^{k_j})\geq g(x_i^\star)\quad i=1,\ldots,N.
    \end{equation}
    By \eqref{eq:xik1}, we get
    \begin{align*}
        \innprod{\nabla_i f(\bm x^{k,i-1})}{x_i^{k+1}-x_i^k}
        +\tfrac{1}{\gamma_i}\D_{h}(\bm x^{k,i},\bm x^{k,i-1})+g_i(x_i^{k+1})
        &\leq \innprod{\nabla_i f^k(\bm x^{k,i-1})}{x_i^\star-x_i^k}\\
        &~~~ +\tfrac{1}{\gamma_i}\D(\bm x^\star,\bm x^{k,i-1})+g_i(x_i^\star).
    \end{align*}
    Using multi-block local strong convexity of $h$ around $\bm x^\star$ and invoking \Cref{pro:suffDecreaseIneq2}, there exist a neighborhood $\mathbf{B}(x_i^\star,\varepsilon_i^\star)$ for $\varepsilon_i^\star>0$, $\sigma_i^{\star}>0$, and $k_i^0\in\N$ such that for $k\geq k_i^0$ and $k\in\mathcal{J}$
    \begin{equation}\label{eq:xk1xk0}
        \lim_{k\to\infty} \tfrac{\sigma_i^{\star}}{2} \|x_i^{k+1}-x_i^k\|^2 \leq\lim_{k\to\infty} \D_{h}(\bm x^{k,i},\bm x^{k,i-1})=0, \quad x_i^k\in \mathbf{B}(x_i^\star,\varepsilon_i^\star),~ i=1,\ldots,N.
    \end{equation}
This indicates that the distance between two successive iterations goes to zero for large enough $k$. Since the sequence $\seq{\bm x^k}$ is bounded, $\nabla f$ and $h$ are continuous, substituting $k=k_j-1$ for $j\in\mathcal{J}$, taking the limit from both sides of the last inequality as $k\to\infty$, and \eqref{eq:xk1xk0}, we come to
    \begin{align*}
        \limsup_{j\to\infty} g_i(x_i^{k_j})\leq g_i(x_i^\star)\quad  i=1,\ldots,N,
    \end{align*}
    and consequently,
    \begin{align*}
        \lim_{j\to\infty} \varphi(\bm x^{k_j})= \lim_{j\to\infty} \left(f(x_1^{k_j},\ldots,x_N^{k_j})+\sum_{i=1}^N g_i(x_i^{k_j})\right)
        = f(x_1^\star,\ldots,x_N^\star)+\sum_{i=1}^N g_i(x_i^\star).
    \end{align*}
    Further, \Cref{pro:suffDecreaseIneq2} and \Cref{pro:subgradLowBound1} ensure $\left(\mathcal{G}_1^{k+1},\ldots,\mathcal{G}_N^{k+1}\right)\in \partial\varphi(\bm x^{k+1})$ and 
    \begin{align*}
        \lim_{k\to+\infty} \|\mathcal{G}_1^{k+1},\ldots,\mathcal{G}_N^{k+1}\|\leq \overline c \lim_{k\to+\infty} \sum_{i=1}^{N}\|x_i^{k+1}-x_i^k)\|
        \leq \overline c \lim_{k\to+\infty} \left(\sum_{i=1}^{N}\tfrac{2}{\sigma_i^\star}\D_{h}(\bm x^{k,i},\bm x^{k,i-1})\right) = 0,
    \end{align*}
    i.e., $\lim_{k\to\infty}\left(\mathcal{G}_1^{k+1},\ldots,\mathcal{G}_N^{k+1}\right)=(0_{n_1},\ldots,0_{n_N})$. Since the subdifferential mapping $\partial \varphi$ is closed, we have $(0_{n_1},\ldots,0_{n_N})\in\partial \varphi (x_1^\star,\ldots,x_N^\star)$, giving \Cref{pro:chatClusterPoint11}.
    
    \Cref{pro:chatClusterPoint12} is a direct consequence of \Cref{pro:chatClusterPoint11}, and \Cref{pro:chatClusterPoint13} and \Cref{pro:chatClusterPoint14} can be proved in the same way as \cite[Lemma 5(iii)-(iv)]{bolte2014proximal}.
\end{proof}

\subsection{Global convergence under Kurdyka-{\L}ojasiewicz inequality} \label{sec:globalConv1}
This section is devoted to the global convergence of \refBPALM[] under Kurdyka-{\L}ojasiewicz inequality.

\begin{defin}[KL property]\label{def:klFunctions1}
A proper and lsc function $\func{\varphi}{\R^{n_1}\times\ldots\times\R^{n_N}}{\Rinf}$ has the \DEF{Kurdyka-{\L}ojasiewicz} property 
(KL property) at $\bm x^\star\in\dom \varphi$ if there exist a concave \DEF{desingularizing function} 
$\psi:[0,\eta]\to{[0,+\infty[}$ (for some $\eta>0$) and neighborhood $\ball{\bm x^\star}{\varepsilon}$ with $\varepsilon>0$, such that
\begin{enumerate}
\item $\psi(0)=0$;
\item $\psi$ is of class $\mathcal{C}^1$ with $\psi>0$ on $(0,\eta)$;
\item for all $\bm x\in \ball{\bm x^\star}{\varepsilon}$ such that $\varphi(\bm x^\star)<\varphi(\bm x)<\varphi(\bm x^\star)+\eta$ it holds that
\begin{equation}\label{eq:klProperty1}
    \psi'(\varphi(\bm x)-\varphi(\bm x^\star))\dist(0,\partial \varphi(\bm x))\geq 1.
\end{equation}
\end{enumerate}
The set of all functions satisfying these conditions is denoted by $\Psi_\eta$.
\end{defin}

The first inequality of this type is given in the seminal work of {\L}ojasiewicz  \cite{lojasiewicz1963propriete,lojasiewicz1993geometrie} for analytic functions, which we nowadays call {\L}ojasiewicz's gradient inequality. Later, Kurdyka \cite{kurdyka1998gradients} showed that this in equality is valid for $\mathcal{C}^1$ functions whose graph belong to an \DEF{$o$-minimal structure} (see its definition in \cite{van1998tame}). The first extensions of the KL property to nonsmooth functions was given by Bolte et al. \cite{bolte2007clarke,bolte2007lojasiewicz,bolte2010characterizations}.

The following two facts constitutes the crucial steps toward the establishment of the global convergence of the sequence generated by \refBPALM[]. 

\begin{fact}[uniformized KL property] 
\cite[Lemma 6]{bolte2014proximal}\label{fac:unifKLProp1}
    Let $\Omega$ be a compact set and $\func{\zeta}{\R^d}{\Rinf}$ be a proper and lower semicontinuous functions. Assume that $\zeta$ is constant on $\Omega$ and satisfies the KL property at each point of $\Omega$. Then, there exists a $\varepsilon>0$, $\eta>0$, and $\psi\in\Psi_\eta$ such that for $\overline u$ and all $u$ in the intersection 
    \[
        \{u\in\R^d~\mid~\dist(u,\Omega)<\varepsilon\}\cap [\zeta(\overline u)<\zeta(u)<\zeta(\overline u)+\eta]
    \]
   we have
    \[
        \psi'(\zeta(u)-\zeta(\overline u))\dist(0,\partial \zeta(u))\geq 1.
    \]
\end{fact}

\begin{fact}\cite[Lemma 2.3]{bot2016Tseng}
\label{lem:convSumTowSeq}
    Let $\seq{a_k}$ and $\seq{b_k}$ be the sequences in $[0,+\infty)$ such that $\sum_{k=1}^\infty b_k<\infty$ and $a_{k+1}=\alpha a_k+b_k$ for all $k\in\N$ in which $\alpha<1$. Then, $\sum_{k=1}^\infty a_k<\infty$.
\end{fact}

Our subsequent main result indicates that the sequence $\seq{\bm x^k}$ generated by \refBPALM[] converges to a critical point $x^\star$ of $\varphi$ if it satisfies the KL property; cf. \Cref{def:klFunctions1}.

\begin{thm}[global convergence]
\label{thm:globConvergence}
    Let \Cref{ass:basic:fgh} hold, let the kernels $h$ be multi-block globally strongly convex with modulus $\sigma_i$ ($i=1,\ldots,N$), and let $\seq{\bm x^k}$ be generated by \refBPALM[] that we assume to be bounded. If $\psi$ is a KL function, then the following statements are true:
    \begin{enumerate}
        \item \label{thm:globConvergence1}
        The sequence $\seq{\bm x^k}$ has finite length, i.e.,
        \begin{equation}\label{eq:finiteLength1}
            \sum_{k=1}^\infty \|x_i^{k+1}-x_i^k\|<\infty \quad i=1,\ldots,N;
        \end{equation}
        \item \label{thm:globConvergence2} 
        The sequence $\seq{\bm x^k}$ converges to a stationary point $\bm x^\star$ of $\varphi$.
    \end{enumerate}
\end{thm}

\begin{proof}
    Let us define the sequence $\seq{\mathcal{S}_k}$ given by $\mathcal{S}_k:=\varphi(\bm x^k)-\varphi^\star$, which is decreasing by \Cref{pro:suffDecreaseIneq1}, i.e., $\seq{\mathcal{S}_k}\to 0$. We now consider two cases: (i) there exists $\overline k\in\N$ such that $\mathcal{S}_k=0$; (ii) $\mathcal{S}_k>0$ for all $k\geq 1$. 
    
    In Case (i), invoking \Cref{pro:suffDecreaseIneq1} implies that $\varphi(\bm x^k)=\varphi^\star$ for all $k\geq \overline k$. It follows from \Cref{pro:suffDecreaseIneq2} and multi-block strong convexity of $h$ that 
    \[
        \tfrac{\sigma_i}{2}\|x_i^{k+1}-x_i^k\|\leq \D_{h}(\bm x^{k,i},\bm x^{k,i-1})=0 \quad i=1,\ldots,N,
    \]
    implying $\bm x^{k+1}=\bm x^{k}$ for all $k\geq \overline k$, which leads to \Cref{thm:globConvergence1}. 
    
    In Case (ii), it holds that $\varphi(\bm x^k)>\varphi^\star$ for all $k\geq 1$. From \Cref{pro:chatClusterPoint13}, the set of limit points $\omega(\bm x^0)$ of $\seq{\bm x^k}$ is nonempty and compact and $\varphi$ is finite and constant on $\omega(\bm x^0)$ due to \Cref{pro:chatClusterPoint14}. Moreover, the sequence $\seq{\varphi(\bm x^k)}$ is decreasing (\Cref{pro:suffDecreaseIneq1}), i.e., for $\eta>0$, there exists a $k_1\in\N$ such that $\varphi^\star<\varphi(\bm x^k)<\varphi^\star+\eta$ for all $k\geq k_1$. For $\varepsilon>0$, \Cref{pro:chatClusterPoint12} implies that there exists $k_2\in\N$ such that $\dist(\bm x^k,\omega(\bm x^0))<\varepsilon$ for $k\geq k_2$. Setting $k_0:=\max\set{k_1,k_2}$ and according to \cref{fac:unifKLProp1}, there exist $\varepsilon, \eta>0$ and a desingularization function $\psi$ such that for any element in 
    \[
        \{\bm x^k~\mid~\dist(\bm x^k,\omega(\bm x^0))<\varepsilon\}\cap [\varphi^\star<\varphi(\bm x^k)<\varphi^\star+\eta] \quad \text{ for } k\geq k_0,
    \]
    the following inequality holds:
    \[
        \psi'(\varphi(\bm x^k)-\varphi^\star)\dist(0,\partial \varphi(\bm x^k))\geq 1 \quad \text{ for } k\geq k_0.
    \]
    Let us define $\Delta_k:=\psi(\varphi(\bm x^k)-\varphi^\star)=\psi(\mathcal{S}_k)$. Then, it follows from the concavity of $\psi$ and \Cref{pro:subgradLowBound1} that
    \begin{align*}
        \Delta_k-\Delta_{k+1}&= \psi(\mathcal{S}_{k})-\psi(\mathcal{S}_{k+1})
        \geq \psi'(\mathcal{S}_{k})(\mathcal{S}_{k}-\mathcal{S}_{k+1})
        =\psi'(\mathcal{S}_{k})(\varphi(\bm x^k)-\varphi(\bm x^{k+1}))\\
        &\geq \frac{\varphi(\bm x^k)-\varphi(\bm x^{k+1})}{\dist(0,\partial \psi(\bm x^k))}
        \geq \frac{\rho \sum_{i=1}^N \D_{h}(\bm x^{k,i},\bm x^{k,i-1})}{\overline c \sum_{i=1}^N \|x_i^{k}-x_i^{k-1}\|}
        \geq \frac{1}{\widehat c} ~\frac{\sum_{i=1}^N \|x_i^{k+1}-x_i^{k}\|^2}{\sum_{i=1}^N \|x_i^{k}-x_i^{k-1}\|},
    \end{align*}
with $\widehat c:=\nicefrac{\overline c}{(\rho \min\set{\sigma_1,\ldots,\sigma_N})}$. Using the arithmetic and quadratic means inequality, and applying the arithmetic and geometric means inequality, it can be concluded that
    \begin{equation}\label{eq:ineqakbk1}
        \begin{split}
            \sum_{i=1}^N \|x_i^{k+1}-x_i^{k}\| \leq \sqrt{\widehat c N (\Delta_k-\Delta_{k+1}) \sum_{i=1}^N \|x_i^{k}-x_i^{k-1}\|}
            \leq \frac{1}{2} \sum_{i=1}^N \|x_i^{k}-x_i^{k-1}\|+\frac{\widehat c N}{2} \left(\Delta_k-\Delta_{k+1}\right).
        \end{split}
    \end{equation}
    We now define the sequences $\seq{a_k}$ and $\seq{b_k}$ as
    \begin{equation}\label{eq:akbk1}
        a_{k+1}:=\sum_{i=1}^N \|x_i^{k+1}-x_i^{k}\|,\quad b_k=\tfrac{\widehat c N}{2} \left(\Delta_k-\Delta_{k+1}\right),\quad \alpha:=\tfrac{1}{2},
    \end{equation}
    where
    $
        \sum_{i=1}^\infty b_k= \tfrac{\widehat cN}{2} \sum_{i=1}^\infty \left(\Delta_i-\Delta_{i+1}\right)
        = \Delta_1-\Delta_\infty = \Delta_1<\infty
    $.
    According to \Cref{lem:convSumTowSeq}, we infer $\sum_{k=1}^\infty a_k<\infty$, which proves \Cref{thm:globConvergence1}.
    
    By \eqref{eq:finiteLength1}, the sequence $\seq{\bm x^k}$ is a Cauchy sequence, i.e., it converges to a stationary point $\bm x^\star$, giving the desired result.
\end{proof}

\begin{rem}\label{rem:seqBoundedness}
In Theorem \ref{pro:chatClusterPoint10} and Theorem \ref{thm:globConvergence}, we implicitly assume that the sequence $\seq{\bm x^k}$ is bounded. This assumption is typical in convergence analysis of proximal-type algorithms for solving general non-convex non-smooth composite optimization problem, see e.g.,  \cite{attouch2010proximal,bolte2018first}. Proposition \ref{pro:suffDecreaseIneq} shows that $\varphi(\bm x^k)$ is non-increasing; hence, it is upper bounded by $\varphi (\bm x^0)$. Therefore, the sequence $\seq{\bm x^k}$ would be bounded if $f(\cdot)$ has bounded level sets and $\sum_{i=1}^n g_i(x_i)$ is bounded below. 
\end{rem}

\subsection{Convergence rate under {\L}ojasiewicz-type inequality} \label{sec:convRate}
We now investigate the convergence rate of the sequence generated by \refBPALM[] under KL inequality of {\L}ojasiewicz type at $x^\star$ ($\psi(s):=\frac{\kappa}{1-\theta} s^{1-\theta}$ with $\theta\in [0,1)$), i.e., there exists $\varepsilon>0$ such that
\begin{equation}\label{eq:LojaKL1}
    |\varphi(\bm x)-\varphi^\star|^{\theta}\leq \kappa \dist(0,\partial \varphi(\bm x)) \quad \forall \bm x\in \ball{\bm x^\star}{\varepsilon}.
\end{equation}

The following fact plays a key role in studying the convergence rate of the sequence generated by \refBPALM[], where its proof can be found in \cite[Lemma 1]{artacho2018accelerating} and \cite[Lemma 12]{bot2019proximal}.

\begin{fact}[convergence rate of a sequence with positive elements]
\label{fac:convRate1}
    Let $\seq{s_k}$ be a sequence in $\R_+$ and let $\alpha$ and $\beta$ be some positive constants. Suppose that $s_k\to 0$ and that the sequence satisfies $s_k^\alpha\leq\beta(s_k-s_{k+1})$ for all $k$ sufficiently large.
    Then, the following assertions hold:
    \begin{enumerate}
        \item If $\alpha=0$, the sequences $\seq{s_k}$ converges to $0$ in a finite number of steps;
        \item If $\alpha\in(0,1]$, the sequences $\seq{s_k}$ converges linearly to $0$ with rate $1-\nicefrac{1}{\beta}$, i.e., there exist $\lambda>0$ and $\tau\in[0,1)$ such that
        \[
            0\leq s_k\leq \lambda \tau^k;
        \]
        \item If $\alpha>1$, there exists $\mu>0$ such that for all $k$ sufficiently large 
        \[
            0\leq s_k\leq \mu k^{-\tfrac{1}{\alpha-1}}.
        \]
    \end{enumerate}
\end{fact}

We next derive the \DEF{convergence rates} of the sequences $\seq{\bm x^k}$ and $\seq{\varphi(\bm x^k)}$ under an additional assumption that the function $\varphi$ satisfies the KL inequality of {\L}ojasiewicz type. 

\begin{thm}[convergence rate]
\label{thm:convRate1}
    Let \Cref{ass:basic:fgh} hold, let the kernel $h$ be multi-block globally strongly convex with modulus $\sigma_1,\ldots,\sigma_N$, and let the sequence $\seq{\bm x^k}$ generated by \refBPALM[] converges to $\bm x^\star$. If $\psi$ satisfies KL inequality of {\L}ojasiewicz type \eqref{eq:LojaKL1}, then the following assertions hold:
    \begin{enumerate}
        \item If $\theta=0$, then the sequences $\seq{\bm x^k}$ and $\seq{\varphi(\bm x^k)}$ converge in a finite number of steps to $\bm x^\star$ and $\varphi(\bm x^\star)$, respectively;
        \item If $\theta\in(0,\nicefrac{1}{2}]$, then there exist $\lambda_1>0$, $\mu_1>0$, $\tau, \overline \tau\in [0,1)$, and $\overline{k}\in\N$ such that 
        \[
            0\leq \|\bm x^k-\bm x^\star\|\leq \lambda_1 \tau^k, \quad 0\leq \mathcal{S}_k\leq \mu_1 \overline\tau^k\quad \forall k\geq \overline{k};
        \]
        \item If $\theta\in(\nicefrac{1}{2},1)$, then there exist $\lambda_2>0$, $\mu_2>0$, and $\overline{k}\in\N$ such that
        \[
            0\leq \|\bm x^k-\bm x^\star\|\leq \lambda_2 k^{-\tfrac{1-\theta}{2\theta-1}}, \quad 0\leq \mathcal{S}_k\leq \mu_2 k^{-\tfrac{1-\theta}{2\theta-1}} \quad \forall k\geq \overline{k}+1.
        \]
    \end{enumerate}
\end{thm}

\begin{proof}
    The proof has two key parts.
    
    In the first part, we show that there exists $\overline k\in\N$ such that for all $k\geq\overline k$ the following inequalities hold for $i=1,\ldots,N$:
    \begin{equation}\label{eq:xkUpperBound}
        \|x_i^k-x_i^\star\|\leq \left\{
        \begin{array}{ll}
            c \max\set{1,\tfrac{\kappa}{1-\theta}} \sqrt{\mathcal{S}_{k-1}} &~~~ \mathrm{if}\  \theta\in[0,\nicefrac{1}{2}], \vspace{1mm}\\
            c \tfrac{\kappa}{1-\theta} \mathcal{S}_{k-1}^{1-\theta} &~~~ \mathrm{if}\ \theta\in(\nicefrac{1}{2},1].
        \end{array}
        \right. 
    \end{equation}
    Let $\varepsilon>0$ be as described in \eqref{eq:LojaKL1} and $x^k\in \ball{x^\star}{\varepsilon}$ for all $k\geq\tilde k$ and $\tilde k\in\N$. By the definitions of $a_k$ and $b_k$ in \eqref{eq:akbk1} and using \eqref{eq:ineqakbk1}, we get $a_{k+1}\leq \tfrac{1}{2} a_k+b_k$ for all $k\geq\tilde k$. Since $\seq{\varphi}$ is nonincreasing, 
    \begin{align*}
        \sum_{i=k}^\infty a_{i+1}\leq \tfrac{1}{2} \sum_{i=k}^\infty (a_i-a_{i+1}+a_{i+1})+ \tfrac{\widehat c}{2}\sum_{i=k}^\infty  \left(\Delta_i-\Delta_{i+1}\right)= \tfrac{1}{2}\sum_{i=k}^\infty a_{i+1}+\tfrac{1}{2} a_k+\tfrac{\widehat c}{2} \Delta_k.
    \end{align*}
    Together with the arithmetic and quadratic means inequality, $\psi(\mathcal{S}_{k})\leq \psi(\mathcal{S}_{k-1})$, and \Cref{pro:suffDecreaseIneq1}, this lead to 
    \begin{equation}\label{eq:sumak1}
        \begin{split}
            \sum_{i=k}^\infty a_{k+1}&\leq a_k+\widehat c \Delta_k = \sum_{i=1}^N \|x_i^{k}-x_i^{k-1}\|+\widehat c \psi(\mathcal{S}_k)
            \leq \sqrt{N} \sqrt{\sum_{i=1}^N \|x_i^{k}-x_i^{k-1}\|^2}+\widehat c \psi(\mathcal{S}_k)\\
            &\leq \sqrt{2N}\max\set{\tfrac{1}{\sqrt{\sigma_1}},\ldots,\tfrac{1}{\sqrt{\sigma_N}}}
            \sqrt{\sum_{i=1}^N \D(\bm x^{k-1,i},\bm x^{k-1,i-1})} +\widehat c \psi(\mathcal{S}_k)\\
            &\leq \sqrt{\tfrac{2N}{\rho}} \max\set{\tfrac{1}{\sqrt{\sigma_1}},\ldots,\tfrac{1}{\sqrt{\sigma_N}}} \sqrt{\mathcal{S}_{k-1}-\mathcal{S}_{k}}+\widehat c \psi(\mathcal{S}_{k-1}).
        \end{split}
    \end{equation}
    On the other hand, for $i=1,\ldots,N$, we have
    \begin{align*}
        \|x_i^k-x_i^\star\|\leq\|x_i^{k+1}-x_i^k\|+\|x_i^{k+1}-x_i^\star\|\leq\ldots\leq\sum_{j=k}^\infty \|x_i^{j+1}-x_i^j\|.
    \end{align*}
    This inequality, together with \eqref{eq:sumak1}, yields
    \[
        \sum_{i=1}^N\|x_i^k-x_i^\star\|\leq \sqrt{\tfrac{2N}{\rho}} \max\set{\tfrac{1}{\sqrt{\sigma_1}},\ldots,\tfrac{1}{\sqrt{\sigma_N}}} \sqrt{\mathcal{S}_{k-1}-\mathcal{S}_{k}}+\widehat c \psi(\mathcal{S}_{k-1}),
    \]
    leading to 
    \begin{equation}\label{eq:xkykUpperBound}
        \|x_i^k-x_i^\star\|\leq c \max\set{\sqrt{\mathcal{S}_{k-1}},\psi(\mathcal{S}_{k-1})}\quad i=1,\ldots,N,
    \end{equation}
    where $c:=\sqrt{\tfrac{2N}{\rho}} \max\set{\nicefrac{1}{\sqrt{\sigma_1}},\ldots,\nicefrac{1}{\sqrt{\sigma_N}}}+\widehat c$ and $\psi(s):=\frac{\kappa}{1-\theta} s^{1-\theta}$. Let us consider the nonlinear equation
    \[
        \sqrt{\mathcal{S}_{k-1}}-\frac{\kappa}{1-\theta} \mathcal{S}_{k-1}^{1-\theta}=0,
    \]
    which has a solution at  $\mathcal{S}_{k-1}=\left(\nicefrac{\kappa}{1-\theta}\right)^{\tfrac{2}{2\theta-2}}$. For $\hat k\in\N$ and $k\geq \hat k$, we assume that \eqref{eq:xkykUpperBound} holds and 
    \[
        \mathcal{S}_{k-1}\leq \left(\frac{\kappa}{1-\theta}\right)^{\tfrac{2}{2\theta-2}}.
    \]
    We now consider two cases: (a) $\theta\in[0,\nicefrac{1}{2}]$; (b) $\theta\in(\nicefrac{1}{2},1]$. In Case (a), if $\theta\in[0,\nicefrac{1}{2})$, then $\psi(\mathcal{S}_{k-1})\leq\sqrt{\mathcal{S}_{k-1}}$. If $\theta=\nicefrac{1}{2}$, then $\psi(\mathcal{S}_{k-1})=\tfrac{\kappa}{1-\theta}\sqrt{\mathcal{S}_{k-1}}$, i.e., $\max\set{\sqrt{\mathcal{S}_{k-1}},\psi(\mathcal{S}_{k-1})}=\max\set{1,\tfrac{\kappa}{1-\theta}} \sqrt{\mathcal{S}_{k-1}}$. Therefore, it holds that $\max\set{\sqrt{\mathcal{S}_{k-1}},\psi(\mathcal{S}_{k-1})}\leq\max\set{1,\tfrac{\kappa}{1-\theta}} \sqrt{\mathcal{S}_{k-1}}$. In Case (b), we have that 
    \[
    \psi(\mathcal{S}_{k-1})\geq\sqrt{\mathcal{S}_{k-1}},
    \]
    i.e., $\max\set{\sqrt{\mathcal{S}_{k-1}},\psi(\mathcal{S}_{k-1})}= \tfrac{\kappa}{1-\theta} \mathcal{S}_{k-1}^{1-\theta}$. Then, it follows from \eqref{eq:xkykUpperBound} that  \eqref{eq:xkUpperBound} holds for all $k\geq\overline k:=\max \set{\tilde k, \hat k}$.
    
    In the second part of the proof, we will show the assertions in the statement of the theorem. For $(\mathcal{G}_i^{k},\ldots,\mathcal{G}_N^{k})\in\partial \varphi(\bm x^k)$ as defined in \Cref{pro:subgradLowBound1}, by \Cref{pro:suffDecreaseIneq1}, we infer
    \begin{align*}
        \mathcal{S}_{k-1}&-\mathcal{S}_{k}=\varphi(x^{k-1})-\varphi(x^{k}) \geq \rho \sum_{i=1}^N \D_{h}(x^{k-1,i},x^{k-1,i-1})\geq \frac{\rho}{2} \sum_{i=1}^N \sigma_i \|x_i^{k}-x_i^{k-1}\|^2\\
        &\geq \frac{\rho}{nN}\min\set{\sigma_1,\ldots,\sigma_N} \left(\sum_{i=1}^N \|x_i^{k}-x_i^{k-1}\|\right)^2
        \geq \frac{\rho}{2N\overline c^2}\min\set{\sigma_1,\ldots,\sigma_N} \|(\mathcal{G}_i^{k},\ldots,\mathcal{G}_N^{k})\|^2\\
        &\geq \frac{\rho}{2N\overline c^2}\min\set{\sigma_1,\ldots,\sigma_N}\dist(0,\partial \varphi(x^k))^2
        \geq \frac{\rho}{2N\overline c^2\kappa^2}\min\set{\sigma_1,\ldots,\sigma_N} \mathcal{S}_{k}^{2\theta}=\widehat{c}~ \mathcal{S}_{k}^{2\theta},
    \end{align*}
   with $\widehat{c}:=\frac{\rho}{2N\overline c^2\kappa^2}\min\set{\sigma_1,\ldots,\sigma_N}$ and for all $k\geq\overline k$. Hence, all assumptions of \cref{fac:convRate1} hold with $\alpha=2\theta$. Therefore, our results follows from this fact and \eqref{eq:xkUpperBound}.
\end{proof}

\subsection{Adaptive BPALM}
The tightness of the $i$-th block upper estimation of the function $f$ given in \Cref{fac:relSmoothEqvi2} is dependent on the parameter $L_i>0$; however, in general, this parameter is a global information and it might not be tight locally, i.e., one may find a $L_i\geq \overline{L}_i(\bm x)\geq 0$ such that 
\[
     f(\bm x+U_i(y_i-x_i))\leq f(\bm x)+\innprod{\nabla_i f(\bm x)}{y_i-x_i}+\overline{L}_i(\bm x) \D(\bm x+U_i(y_i-x_i),\bm x)
\]
for all $\bm y\in\ball{\bm x}{\varepsilon_1}$ with a small enough $\varepsilon_1>0$. Consequently, the majorization model described by $\M$ may not be tight enough, which will consequently lead to smaller stepsizes $\gamma_i\in(0,\nicefrac{1}{L_i})$. In this case and in the case that $L_1,\ldots,L_N$ are not available, one can retrieve them adaptively by applying a backtracking linesearch starting from a lower estimates; see, e.g., \cite{ahookhosh2019accelerated,ahookhosh2019bregman,mukkamala2019convex,nesterov2013gradient,themelis2018forward}. 

Putting together the above discussions, we propose an adaptive version of \refBPALM[] using a backtracking linesearch; see Algorithm~\ref{alg:abpalm}. 

\begin{algorithm*}[ht] 
\algcaption{({\bf A-BPALM}) adaptive BPALM}%
\begin{algorithmic}[1]
\Require{%
	\(\bm x^0\in\R^{n_1}\times\ldots\times\R^{n_N}\),~ $\nu_1>1$,~$L_i\geq \overline{L}_i^0>0$ for $i=1,\ldots,N$,\ $I_n=(U_1,\ldots,U_N)\in\R^{n\times n}$ with $U_i\in\R^{n\times n_i}$ and the identity matrix $I_n$.%
}%
\Initialize{$k=0$,~$p=0$,~ $\gamma_i^0\in\left(0,\nicefrac{1}{\overline{L}_i^0}\right)$ for $i=1,\ldots,N$.}%
\While{some stopping criterion is not met}
\State\label{state:pfbpalmm0Xk1}%
    $\bm x^{k,0}=\bm x^k$;
    \For{\texttt{$i=1,\ldots,N$}}
        \Repeat
            \State\label{state:abpalmLiP}%
            set $\overline{L}_i^{k+1}=\nu_1^p \overline{L}_i^{k}, \quad \gamma_i^{k+1}=\nicefrac{\gamma_i^k}{\nu_1^p},\quad p=p+1;$
            \State\label{state:abpalmmTi}%
            compute $x_i^{k,i}\in \T[][\nicefrac{h}{ \gamma_i^{k+1}}](\bm x^{k,i-1}),\quad \bm x^{k,i}=\bm x^{k,i-1}+U_i(x_i^{k,i}-x_i^{k,i-1})$;
        \Until{$f(\bm x^{k,i})\leq f(\bm x^{k,i-1})+\innprod{\nabla_i f(\bm x^{k,i-1})}{\bm x_i^{k,i}-\bm x_i^{k,i-1}}+\overline{L}_i^{k+1}\D(\bm x^{k,i},\bm x^{k,i-1})$}
        \State\label{state:abpalmpiLik1}%
        $p= 0$;
    \EndFor\label{state:abpalmLineSearch}
\State\label{state:pfbpalmk1}%
	 $\bm x^{k+1}=\bm x^{k,N}$,~$k= k+1$;%
\EndWhile
\Ensure{%
	A vector $\bm x^k$.%
}%
\end{algorithmic}
\label{alg:abpalm}
\end{algorithm*}%

We next provide an upper bound on the \DEF{total number of calls of oracle} after $k$ iterations of \refaBPALM[] and those needed to satisfy \eqref{eq:stopCrit}.

\begin{prop}[worst-case oracle calls]
\label{pro:worstCaseOracle}
    Let $\seq{\bm x^k}$ be generated by \refaBPALM[]. Then, 
    \begin{enumerate}
    \item \label{pro:worstCaseOracle1}
        after at most
        $\tfrac{1}{\ln \nu_1}\left(\ln (\nu_1 L_i)-\ln \overline{L}_i^0\right)$ 
        iterations the linesearch (Lines 4 to 7 of \refaBPALM[]) is terminated;
        \item \label{pro:worstCaseOracle2}
        the number of oracle call after $k$ full cycle $\mathcal{N}_k$ is bounded by
        \[
        		\mathcal{N}_k\leq 2N(k+1)+\tfrac{2}{\ln \nu_1}\sum_{i=1}^N \ln \tfrac{\nu_1L_i}{\overline{L}_i^0};
        \]
        \item \label{pro:worstCaseOracle3}
            the worst-case number of oracle calls to satisfy \eqref{eq:stopCrit} is given by
            \[
            \mathcal{N}_k\left( 1+\tfrac{ (\varphi(\bm x^0)-\inf \varphi)}{\overline \rho \varepsilon}\right),
            \]
            with $ \overline\rho:=\min\set{\nicefrac{(1-\gamma_1^0 \overline L_1^0)}{\gamma_1^0},\ldots,\nicefrac{(1-\gamma_N^0 \overline L_N^0)}{\gamma_N^0}}$.
    \end{enumerate}
\end{prop}

\begin{proof}   
	According to \cref{state:abpalmLiP} and \cref{state:abpalmpiLik1} of \refaBPALM[], we have $\overline{L}_i^{k+1}=\nu_1^{p_i^k} \overline{L}_i^k$, i.e.,
	\begin{align*}
		p_i^k= \tfrac{1}{\ln \nu_1}\left(\ln \overline{L}_i^{k+1}-\ln \overline{L}_i^k\right)\leq \tfrac{1}{\ln \nu_1}\left(\ln (\nu_1 L_i)-\ln \overline{L}_i^0\right) \quad i=1,\ldots,N,
	\end{align*}
	giving \Cref{pro:worstCaseOracle1}. Hence, the total number of calls of oracle after $k$ iterations is given by
    \begin{align*}
        \mathcal{N}_k&=\sum_{j=0}^k\sum_{i=1}^N 2(p_i^j+1)\leq 2\sum_{i=1}^N\left[(k+1)+\tfrac{1}{\ln \nu_1}\sum_{j=0}^k \left(\ln \overline{L}_i^{j+1}-\ln \overline{L}_i^j\right)\right] \\
        &=2N(k+1)+\tfrac{2}{\ln \nu_1}\sum_{i=1}^N \ln (\overline{L}_i^{k+1})-\ln(\overline{L}_i^0) \leq 2N(k+1)+\tfrac{2}{\ln \nu_1}\sum_{i=1}^N \ln \tfrac{\nu_1L_i}{\overline{L}_i^0}, 
    \end{align*}
    giving \Cref{pro:worstCaseOracle2}. 

	Following the proof of \Cref{pro:suffDecreaseIneq} and since the sequence $\seq{\nicefrac{1-\gamma_1^k \overline L_1^k}{\gamma_1^k}}$ is increasing with respect to $k$, it is easy to see that 
	 \begin{equation*}
               \overline{\rho}\leq \min\set{\tfrac{1-\gamma_1^{k+1} \overline L_1^{k+1}}{\gamma_1^{k+1}},\ldots,\tfrac{1-\gamma_N^{k+1} \overline L_N^{k+1}}{\gamma_N^{k+1}}} \sum_{i=1}^N \D(\bm x^{k,i},\bm x^{k,i-1})\leq \varphi(\bm x^k)-\varphi(\bm x^{k+1}).
      \end{equation*}
	On the other hand,  \cref{state:abpalmLiP} implies that
	$\nicefrac{(1-\gamma_i^{k+1} \overline L_i^{k+1})}{\gamma_i^{k+1}}\geq \nicefrac{(1-\gamma_i^0 \overline L_i^0)}{\gamma_i^0}$, $i=1,\ldots,N$,
	leading to
	\begin{equation}\label{eq:phik1phik}
                \overline\rho \sum_{i=1}^N \D(\bm x^{k,i},\bm x^{k,i-1})\leq \varphi(\bm x^k)-\varphi(\bm x^{k+1}).
      \end{equation}
    Following the proof of \Cref{cor:iterComplexity}, we have that \refBPALM[] will be terminated within $k\leq1+ \tfrac{\varphi(x^0)-\inf \varphi}{\overline\rho \varepsilon}$ iterations. Together with  \Cref{pro:worstCaseOracle2}, this implies that \Cref{pro:worstCaseOracle3} is true.
\end{proof}

Choosing appropriate constants $\overline{L}_1^0,\ldots,\overline{L}_N^0$,  \Cref{pro:worstCaseOracle1} roughly speaking says that on average each full cycle of \refaBPALM[] needs at most $2N$ oracle calls. 
      Furthermore, in light of \eqref{eq:phik1phik}, \Cref{pro:suffDecreaseIneq} holds true by replacing $\rho$ with $\overline \rho$. Considering this replacement, all the results of \Cref{pro:subgradLowBound1},  \Cref{pro:chatClusterPoint10},  \Cref{thm:globConvergence}, and  \Cref{thm:convRate1} remain valid for \refaBPALM[].
      
\begin{rem}[A-BPALM variant]\label{rem:varA-BPALM}
	We here notice that one may change Line 5 of \refaBPALM[] as ``set $\overline{L}_i^{k+1}=\nu_1^p \overline{L}_i^{0},  \gamma_i^{k+1}=\nicefrac{\gamma_i^0}{\nu_1^p}, p=p+1;$", which always start the backtracking procedure from $\overline{L}_i^{0}$ and $\gamma_i^0$. It is easy to see the results of \Cref{pro:worstCaseOracle} are still valid for this variant of \refaBPALM[].
\end{rem}

	\section{Application to orthogonal nonnegative matrix factorization}\label{sec:appONMF}
		A natural way of analyzing large data sets is finding an effective way to represent them using dimensionality reduction methodologies.
\DEF{Nonnegative matrix factorization} (NMF) is one such technique that has received much attention
in the last few years; see, e.g., \cite{cichocki2009nonnegative, fu2018nonnegative,gillis2014and} and the references therein. In order to extract hidden and important features from data, NMF decomposes the data matrix into two factor matrices (usually much smaller than the original data matrix) by imposing componentwise nonnegativity and (possibly) sparsity constraints on these factor matrices.
More precisely, let the data matrix be $X=[x_1,x_2,\ldots,x_n]\in\R_+^{m\times n}$ where each $x_i$ represents some data point. 
NMF seeks a decomposition of $X$ into a nonnegative $n\times r$ basis matrix $U=[u_1,u_2,\ldots,u_r]\in\R_+^{m\times r}$ and a  nonnegative $r\times n$ coefficient matrix $V=[v_1,v_2,\ldots,v_r]^T\in\R_+^{r\times n}$ such that 
\begin{equation}\label{eq:nmfEq}
    X\approx UV,    
\end{equation}
where $\R_+^{m\times n}$ is the set of $m\times n$ element-wise nonnegative matrices. Extensive research has been carried out on variants of NMF, and most studies in this area have focused on algorithmic developments, but with very limited convergence theory. This motivates us to study the application of \refBPALM[] and \refaBPALM[] to a variant of NMF, namely orthogonal NMF (ONMF).

\subsection{Orthogonal nonnegative matrix factorization}\label{sec:onmf}
Besides the decomposition \eqref{eq:nmfEq}, the \DEF{orthogonal nonnegative matrix factorization} (ONMF) involves an additional orthogonality constraint $VV^T=I_r$ leading to the constrained optimization problem
\begin{equation}\label{eq:onmf}
    \begin{array}{ll}
        \minimize &~~ \tfrac{1}{2}\|X-UV\|_F^2 \\
        \stt      &~~ U\geq0,~V\geq0,~VV^T=I_r
    \end{array}
\end{equation}
where $I_r\in\R^{r\times r}$ is the identity matrix. 
By imposing the matrix $V$ to be orthogonal (as well as nonnegative), ONMF imposes that each data points is only associated with one basis vector hence ONMF is closely related to clustering problems; see~\cite{pompili2014two} and the references therein.  
Since the projection onto the set 
$
C:=\set{(U,V)\in\R^{m\times r}\times\R^{r\times n}\mid U\geq0,~V\geq0,~VV^T=I_r}
$
is costly, we here consider the penalized formulation
\begin{equation}\label{eq:penelONMF}
    \begin{array}{ll}
        \minimize &~~ \tfrac{1}{2}\|X-UV\|_F^2+\tfrac{\lambda}{2}\|I_r-VV^T\|_F^2 \\
        \stt      &~~ U\geq0,~V\geq0,
    \end{array}
\end{equation}
for the penalty parameter $\lambda>0$. Introducing a product separable kernel, we next show that the objective function \eqref{eq:penelONMF} is multi-block relatively smooth.

\begin{prop}[multi-block relative smoothness of ONMF objective]
\label{pro:relSmoothNMF0}
    Let the function $\func{h}{\R^{m\times r}\times\R^{r\times n}}{\Rinf}$ be a kernel given by
    \begin{equation}\label{eq:multih4h2}
        h(U,V):=\left(\tfrac{\beta_1}{2}\|U\|_F^2+1\right) \left(\tfrac{\alpha_2}{4} \|V\|_F^4+\tfrac{\beta_2}{2}\|V\|_F^2+1\right). 
    \end{equation}
   Then the function  $\func{f}{\R^{m\times r}\times \R^{r\times n}}{\Rinf}$ given by $f(U, V):=\tfrac{1}{2}\|X-UV\|_F^2+\tfrac{\lambda}{2}\|I_r-VV^T\|_F^2$ is $(L_1,L_2)$-smooth relative to $h$ with
    \begin{equation}\label{eq:l1l2upper}
        L_1\geq \tfrac{2}{\beta_1\beta_2}, \quad L_2\geq 6\max\set{\tfrac{\lambda}{\alpha_2}, \tfrac{2\lambda}{\beta_1\beta_2}, \tfrac{\lambda}{\beta_2}}.
    \end{equation}
\end{prop}

\begin{proof}
    Using partial derivatives $\nabla_U f(U, V)=U VV^T-XV^T$, $\nabla_{UU}^2 f(U, V) Z=Z VV^T$, and the Cauchy Schwarz inequality, it can be concluded that
    $
        \innprod{Z}{\nabla_{UU}^2 f(U, V) Z}\leq \|V\|_F^2 \|Z\|_F^2
   $.
    On the other hand, $\nabla_{U}h(U,V)=\beta_1\left(\tfrac{\alpha_2}{4} \|V\|_F^4+\tfrac{\beta_2}{2}\|V\|_F^2+1\right)U$ and
    \begin{align*}
        \innprod{Z}{\nabla_{UU}^2 h(U, V) Z}= \beta_1\left(\tfrac{\alpha_2}{4} \|V\|_F^4+\tfrac{\beta_2}{2}\|V\|_F^2+1\right)\innprod{Z}{Z}
        \geq \tfrac{\beta_1\beta_2}{2}\|V\|_F^2\|Z\|_F^2.
    \end{align*}
    Together with \eqref{eq:l1l2upper}, this yields 
    \begin{align*}
        \innprod{Z}{(L_1\nabla_{UU}^2 h(U,V)&-\nabla_{UU}^2f(U, V)) Z}
        \geq \left(\tfrac{\beta_1\beta_2}{2}L_1-1\right)\|V\|_F^2\|Z\|_F^2\geq 0,
    \end{align*}
    which implies $L_1\nabla_{UU}^2 h(U,V)-\nabla_{UU}^2f(U, V)\succeq 0$.
    
    From $\nabla_V f(U, \cdot)(V)= U^TUV-U^TX+2\lambda (V V^TV-V)$ and the definition of directional derivative, we obtain
    \begin{align*}
        \nabla_{VV}^2f(U, V) Z&=\lim_{t\to 0} 
        \frac{
           U^TU(V+tZ)-U^TX+2\lambda[(V+tZ)(V+tZ)^T(V+tZ)-(V+tZ)]}{t}\\
        &~~~~~~~~~~-\frac{U^TUV-U^TX+2\lambda (V V^TV-V)}{t}\\   
        &= U^TUZ+2\lambda (ZV^TV+VZ^TV+VV^TZ-Z) \quad \forall Z \in\R^{r\times n}.
    \end{align*}
    This, $\innprod{Y_1}{Y_2}:=\trace(Y_1^TY_2)$, basic properties of the trace, the Cauchy-Schwarz inequality, and the submultiplicative property of the Frobenius norm imply
    \begin{align*}
        \innprod{Z}{\nabla_{VV}^2f(U, V)Z}&=\innprod{Z}{U^TUZ+2\lambda (ZV^TV+VZ^TV+VV^TZ-Z)}\\
        &=\lambda\|ZV^T+VZ^T\|_F^2+2\lambda\innprod{ZZ^T}{VV^T}-2\lambda\|Z\|_F^2+\innprod{Z}{U^TUZ}\\
        &\leq 2\lambda \left(\|Z\|_F^2\|V\|_F^2+\|Z\|_F^2\|V\|_F^2\right)+2\lambda (\|V\|_F^2+\|U\|_F^2+1)\|Z\|_F^2\\
        &\leq 6\lambda \left(\|V\|_F^2+\|U\|_F^2+1\right)\|Z\|_F^2.
    \end{align*}
    Plugging $\nabla_V h(U,V)=\left(\tfrac{\beta_1}{2}\|U\|_F^2+1\right)\left(\alpha_2 \|V\|_F^2+\beta_2\right)V$ into the directional derivative definition, we come to
    \begin{align*}
        \nabla_{VV}^2 h(U,V) Z&=\left(\tfrac{\beta_1}{2}\|U\|_F^2+1\right)~ \lim_{t\to 0} \frac{\left(\alpha_2\|V+tZ\|_F^2+\beta_2\right)(V+tZ)-\left(\alpha_2\|V\|_F^2+\beta_2\right)V}{t}\\
        &= \left(\tfrac{\beta_1}{2}\|U\|_F^2+1\right)\left[\left(\alpha_2 \|V\|_F^2+\beta_2\right)Z+2\alpha_2\innprod{V}{Z}V\right],
    \end{align*}
    implying
    \begin{align*}
        \innprod{Z}{\nabla_{VV}^2 h(U,V) Z}&= \left(\tfrac{\beta_1}{2}\|U\|_F^2+1\right)\left[\left(\alpha_2 \|V\|_F^2+\beta_2\right)\|Z\|_F^2+2\alpha_2\innprod{V}{Z}^2\right]\\
        &\geq \left(\tfrac{\beta_1}{2}\|U\|_F^2+1\right)\left(\alpha_2 \|V\|_F^2+\beta_2\right)\|Z\|_F^2\\
        &\geq \left(\alpha_2 \|V\|_F^2+\tfrac{\beta_1\beta_2}{2}\|U\|_F^2+\beta_2\right)\|Z\|_F^2.
    \end{align*}
    Hence, it follows from \eqref{eq:l1l2upper} that
    \begin{align*}
        \innprod{Z}{(L_2\nabla_{VV}^2 h(U,V)&-\nabla_{VV}^2f(U, V)) Z}\\
        &\geq \left((L_2\alpha_2-6\lambda) \|V\|_F^2+(L_2\tfrac{\beta_1\beta_2}{2}-6\lambda)\|U\|_F^2+(L_2\beta_2-6\lambda)\right)\|Z\|_F^2\geq 0,
    \end{align*}
    i.e., $L_2\nabla_{VV}^2 h(U,V)-\nabla_{VV}^2f(U, V)\succeq 0$, as claimed.
\end{proof}

The unconstrained version of the ONMF problem \eqref{eq:onmf} is given by
\begin{equation}\label{eq:orthNMF2}
        \minimize_{(U,V)} ~~ \tfrac{1}{2}\|X-UV\|_F^2+\tfrac{\lambda}{2}\|I_r-VV^T\|_F^2+\delta_{U\geq 0}+ \delta_{V\geq 0},
\end{equation}
where $\delta_{U\geq 0}$ and $\delta_{V\geq 0}$ are the indicator functions of the sets $C_1:=\set{U\in\R^{m\times r}\mid U\geq 0}$ and $C_2:=\set{V\in\R^{r\times n}\mid V\geq 0}$, respectively. Comparing to \eqref{eq:P}, the next setting is recognized
\[
    f(U,V):=\tfrac{1}{2}\|X-UV\|_F^2+\tfrac{\lambda}{2}\|I_r-VV^T\|_F^2,\quad g_1(U):=\delta_{U\geq 0},\quad g_2(V):= \delta_{V\geq 0},
\]
in which both $g_1(U)$ and $g_2(V)$ are nonsmooth and convex, and $f(U,V)$ is $(L_1,L_2)$-smooth relative to $h$ given in \eqref{eq:multih4h2}; cf. \Cref{pro:relSmoothNMF0}. For given $U^k$ and $V^k$, applying \refBPALM[] and \refaBPALM[] to \eqref{eq:orthNMF2}, $U^{k+1}$ and $V^{k+1}$ should be computed efficiently, which we study next.

\begin{thm}[closed-form solutions of the subproblem \eqref{eq:xik1} for ONMF]
\label{thm:iterBPALM}
    Let $\func{h_1}{\R^{m\times r}}{\Rinf}$ and $\func{h_2}{\R^{r\times n}}{\Rinf}$ be the kernel functions given by
    \[
        h_1(U):=\tfrac{\beta_1}{2}\|U\|_F^2+1, \quad h_2(V):=\tfrac{\alpha_2}{4} \|V\|_F^4+\tfrac{\beta_2}{2}\|V\|_F^2+1,
    \]
    i.e., $h(U,V)=h_1(U)h_2(V)$. For given $U^k$ and $V^k$, the problem \eqref{eq:orthNMF2}, and the subproblem \eqref{eq:xik1}, the following assertions hold:
    \begin{enumerate}
        \item \label{thm:iterBPALM1}
        For  $\eta_1=\tfrac{\alpha_2}{4}\|V^k\|_F^4+\tfrac{\beta_2}{2}\|V^k\|_F^2+1$ and $\mu_1:=\nicefrac{\gamma_1}{(\beta_1\eta_1)}$, the iteration $U^{k+1}$ is given by
            \begin{equation}\label{eq:uk01}
                 U^{k+1}= \max\set{U^k-\mu_1\left(U^k V^k(V^k)^T-X(V^k)^T\right),0};
            \end{equation}
        \item \label{thm:iterBPALM2}
        For $\eta_2=\tfrac{\beta_1}{2}\|U^{k+1}\|_F^2+1$  and $\mu_2:=\nicefrac{\gamma_2}{\eta_2}$, the iteration $V^{k+1}$ is given by
            \begin{equation}\label{eq:vk01}
                V^{k+1}= \tfrac{1}{t_k} \max\set{(\alpha_2 \|V^k\|_F^2+\beta_2)V^k-\mu_2\nabla_V f(U^{k+1},V^k),0}
            \end{equation}
            with $\nabla_V f(U^{k+1},V^k)= (U^{k+1})^TU^{k+1}V^k-(U^{k+1})^TX+2\lambda (V^k (V^k)^TV^k-V^k)$ and
             \begin{align}\label{eq:ckSol}
                t_k = \frac{\beta_2}{3}+\sqrt[3]{-\frac{\tau_2}{2}+\sqrt{\left(\tfrac{\tau_2}{2}\right)^2+\left(\tfrac{\tau_1}{3}\right)^3}}+\sqrt[3]{-\frac{\tau_2}{2}-\sqrt{\left(\tfrac{\tau_2}{2}\right)^2+\left(\tfrac{\tau_1}{3}\right)^3}},
           \end{align}
           where $\tau_1:=-\nicefrac{\beta_2^2}{3}$ and $\tau_2:=\nicefrac{\left(-2\beta_2^3-27\alpha_2\big\|\max\set{(\alpha_2 \|V^k\|_F^2+\beta_2)V^k-\mu_2\nabla_V f(U^{k+1},V^k),0}\big\|_F\right)}{27}$.
    \end{enumerate}
    \end{thm}

\begin{proof}
	Setting $g_1:=\delta_{U\geq 0}$ and $f(U,V)=\tfrac{1}{2}\|X-UV\|_F^2+\tfrac{\lambda}{2}\|I_r-VV^T\|_F^2$, it follows from \eqref{eq:tix} that
	\begin{align*}
		U^{k+1}&=  \argmin_{U\in\R^{m\times r}} \set{\innprod{\nabla_U f(U^k,V^k)}{U-U^k} +\tfrac{1}{\gamma_1}\D((U,V^k),(U^k,V^k))+ g_1(U)}\\
		&=\argmin_{U\in\R^{m\times r}} \set{\tfrac{\beta_1\eta_1}{2\gamma_1}\|U-(U^k-\tfrac{\gamma_1}{\beta_1 \eta_1}\nabla_U f(U^k,V^k))\|_F^2+ g_1(U)}\\
		&=\proj_{U\geq 0}(U^k-\mu_1\nabla_U f(U^k,V^k)),
	\end{align*}
	with $\nabla_U f(U^k,V^k)= U^k V^k(V^k)^T-X(V^k)^T$, giving \eqref{eq:uk01}.
	
	By setting $g_2:=\delta_{V\geq 0}$ and invoking \eqref{eq:tix}, we infer
	\begin{align*}
		V^{k+1}&=  \argmin_{V\in\R^{r\times n}} \set{\innprod{\nabla_V f(U^{k+1},V^k)}{V-V^k} +\tfrac{1}{\gamma_2}\D((U^{k+1},V),(U^{k+1},V^k))+ g_2(V)}\\
		&=  \argmin_{V\in\R^{r\times n}} \set{g_2(V)+h_2(V)-\innprod{\nabla h_2(V^k)-\mu_2\nabla_V f(U^{k+1},V^k)}{V-V^k}}.
	\end{align*}
	Let us consider the normal cone $\mathcal{N}_{V\geq 0}(V^{k+1})=\set{P\in\R^{r\times n}\mid V^{k+1}\odot P=0,\ P\leq 0}$ (see \cite[Corollary 3.5]{tam2017regularity}), where $V\odot P$ denotes the \DEF{Hadamard products} given pointwise by $(V\odot P)_{ij}:=V_{ij}P_{ij}$ for $i\in{1,\ldots,r}$ and $j\in{1,\ldots,n}$. The first-order optimality conditions for the latter identity leads to $G^k-(\alpha_2\|V^{k+1}\|_F^2+\beta_2)V^{k+1}\in\mathcal{N}_{V\geq 0}(V^{k+1})$ with $G^k:=\nabla h_2(V^k)-\mu_2\nabla_V f(U^{k+1},V^k)$. Let us consider two cases: (i) $G_{ij}\leq 0$; (ii) $G_{ij}> 0$. In Case (i), $P_{ij}=G_{ij}^k-(\alpha_2\|V^{k+1}\|_F^2+\beta_2)V_{ij}^{k+1}\leq 0$, i.e., $V_{ij}^{k+1}=0$. In Case (ii), if $V_{ij}^{k+1}=0$, then $P_{ij}=G_{ij}^k>0$, which contradicts $P\leq 0$, i.e.,  $G_{ij}^k-(\alpha_2\|V^{k+1}\|_F^2+\beta_2)V_{ij}^{k+1}=0$.  Combining both cases, we come to the equation
	\[
		(\alpha_2\|V^{k+1}\|_F^2+\beta_2)V^{k+1}=\proj_{V\geq 0}(G^k),
	\]
	i.e., there exists $t_k\in \R$ such that $t_k V^{k+1}=\proj_{V\geq 0}(G^k)$ that eventually lead to
	\begin{align*}
		t_k^3-\beta_2 t_k^2-\alpha_2 \|\proj_{V\geq 0}(G^k)\|_F^2=0,
	\end{align*}
	which is a Cardano equation and its solution is given by \eqref{eq:ckSol}.
\end{proof}

\subsection{Preliminary numerical experiment}\label{sec:numExper}
In this section, we report preliminary numerical experiments with \refBPALM[] and two variants of \refaBPALM[], namely,
\begin{enumerate}
	\item A-BPALM1: the algorithm \refaBPALM[];
	\item A-BPALM2: the variant of \refaBPALM[] as described in \Cref{rem:varA-BPALM}.
\end{enumerate}
Since the unconstrained ONMF problem \eqref{eq:orthNMF2} involves the penalty term $\tfrac{\lambda}{2}\|I_r-VV^T\|_F^2$, we also consider a ``continuation" variant of these algorithm that starts from some $\lambda>0$, run one of the above-mentioned algorithms until some stopping criterion holds and save its best point, and then it increases the penalty parameter and run the algorithm with the starting point as the best point of the last call, and it continues the procedure until we stop the algorithm. We refer to this procedure as \DEF{continuation}, which we will describe next in more details. 

 \vspace{-2mm}
\begin{algorithm*}[] 
\algcaption{Continuation procedure}%
\begin{algorithmic}[1]
\Require{%
	\(\bm x^0\in\R^{n_1}\times\ldots\times\R^{n_N}\),~ $\lambda>0$,~ $c>1$.%
}%
\Repeat
\State\label{state:call}%
    starting from $\bm x^0$; run one of BPALM, A-BPALM1, or A-BPALM2  to attain an inexact solution $ \overline{\bm x}$ of˜\eqref{eq:orthNMF2};
\State\label{state:pfbpalmk1}%
	 set $\bm x^0\gets\overline{\bm x}$,~ $\lambda \gets c \lambda$;%
\Until{some stopping criterion holds}
\Ensure{%
	$\overline{\bm x}$%
}%
\end{algorithmic}
\label{alg:continuation}
\end{algorithm*}%
 \vspace{-2mm}

In our Implementation all the codes were written in MATLAB (publicly available at \url{https://github.com/MasoudAhoo/BPALM}) and runs were performed on a MacBook Pro with 2,8 GHz Intel Core i7 CPU and 16 GB RAM. On the basis of our preliminary experiments, we here set $\alpha_2=\beta_2=\beta_1=1$ to provide the relative smoothness constants as described in \eqref{eq:l1l2upper}, and the related step-sizes are computed by $\gamma_i=\nicefrac{1}{L_i}-\epsilon$, when $\epsilon$ is set as the machine precision.  For A-BPALM1 and A-BPALM1, we set $\nu=2$, and we also set $\overline{L}_i^0=0.01L_i$ for A-BPALM1 and $\overline{L}_i^0=0.1L_i$ for A-BPALM2. For the continuation version, we set $c=\nicefrac{3}{2}$. 

We first report the experiment on a synthetic data set with $(m,n,r)=(200,2000,10)$. Our synthetic data set is generated as follows. We use the MATLAB command $\mathsf{rand}$ to  generate random nonnegative matrices $U \in \mathbb R_+^{m\times r}$ and $R\in \mathbb R_+^{m\times n}$, then we generate a random orthogonal nonnegative matrix $V\in \mathbb R_+^{r\times n}$. Next, we set $X=UV$ to obtain the $m$-by-$n$ orthogonal decomposable matrix $X$, and finally add 5\% of noise by $X=X+0.05 \frac{\| X\|_F}{\|R\|_F } R $. Now, we use SVD-based initialization for providing starting points for our algorithms, see \cite{Boutsidis2008}. We here run our algorithms with both fixed penalty parameter and with the continuation scheme. For fixed penalty versions, we set $\lambda=10$, and for continuation versions we started with $\lambda=10$ and stopped the inner algorithms every 3 seconds and increased $\lambda$ by factor $c=\nicefrac{3}{2}$. We stopped the algorithms after 15 seconds of the running time.

The results of our implementation are illustrated 
in \Cref{fig:syntheticData}. In this figure, Subfigure~(a) stands for fixed penalty versions while Subfigure~(b) stands for continuation versions. Hence, on Subfigure~(b), the penalty $\lambda$ is progressively increased. 
We make two observations: (i)
In both cases A-BPALM1 and A-BPALM2 outperform BPALM while A-BPALM1 is the best among them; (ii) The continuation schemes perform much better than the fixed penalty versions, especially for A-BPALM1. In fact, although the curve on Subfigure~(b) corresponds to a larger value of $\lambda$, A-BPALM1 achieves a much lower function value; namely around 20 on Subfigure~(a) vs.\@ 0.2 on Subfigure~(b). The reason is that increasing $\lambda$ leads to a better solution where the factor $V$ is closer to orthogonality hence closer to the ground truth. 

\begin{figure}[h!]
\centering
\begin{subfigure}[t]{0.45\textwidth}
\centering
{\includegraphics[width=5.7cm]{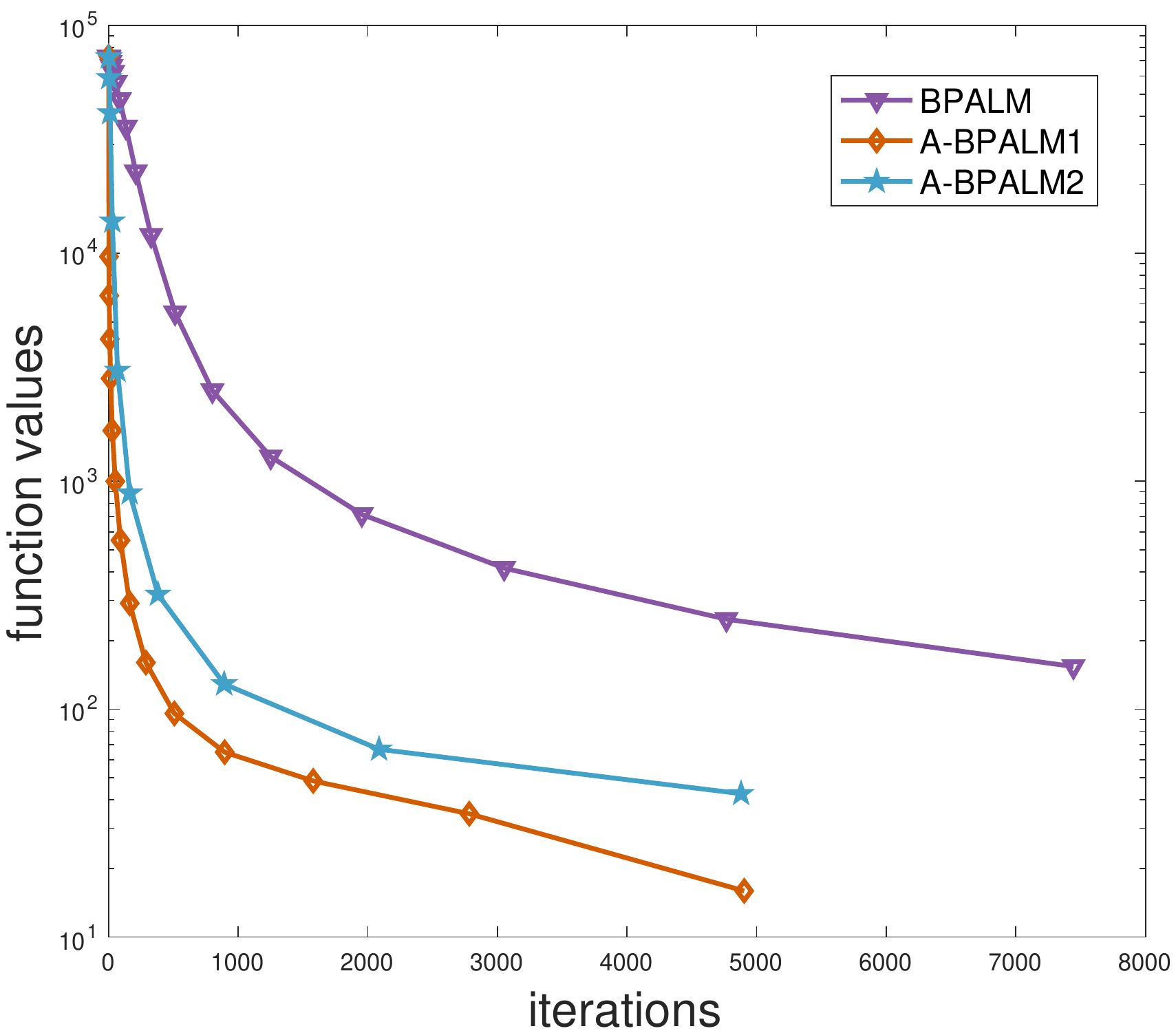}}%
\caption{Algorithms with fixed penalty parameter}
\end{subfigure}
\qquad
\begin{subfigure}[t]{0.45\textwidth}
\centering
{\includegraphics[width=5.7cm]{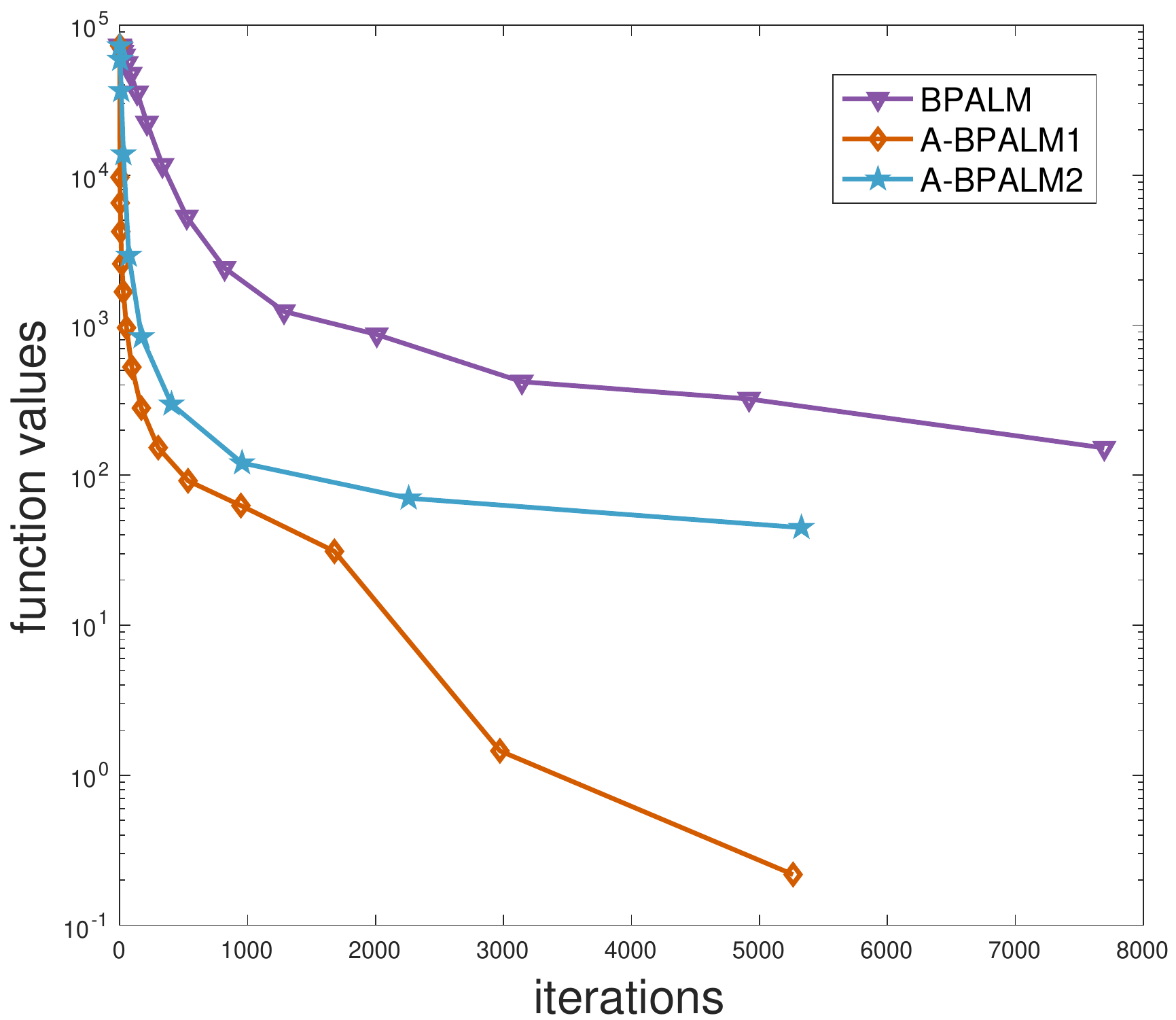}}%
\caption{Algorithms with continuation}
\end{subfigure}
\vspace{-2mm}
\caption{A comparison among BPALM, A-BPALM1, and A-BPALM2 for the synthetic data, where the algorithms stopped after 15 seconds of running time. For algorithms with a fixed penalty parameter, we set $\lambda=10$, and for the algorithms with the continuation procedure, we start from $\lambda=10$ and increase this parameter by factor $3/2$ every 3 seconds. Note that the $y$-axis has different scales on both figures.}
\label{fig:syntheticData}
\end{figure}

We next report the performance of our algorithms on the Hubble telescope data set which is taken from \cite{Pauca2006}. Since the continuation versions of our algorithms perform better, we here only apply the continuation versions of BPALM, A-BPALM1, and A-BPALM2. 
We use the SVD-based initialization as in \cite{pompili2014two}. In this problem, each row of the matrix $X$ is a vectorized image of the Hubble telescope at a given wavelength for a total of $m=100$ wavelengths. Each image contains $n=128\times 128$  pixels. Since each pixel in the image contains mostly a single material, it makes sense to use ONMF to cluster the pixel according to the material they contain (see \Cref{fig:hubbleData} for an illustration). 
For this application problem, we report the final relative fidelity and orthogonal errors, i.e., 
\[
F_{error}:=\nicefrac{\|X-U^kV^k\|_F}{\|X\|_F}, \quad O_{error}:=\|I-V^k(V^k)^T\|_F,
\]
with respect to several initial values for the penalty parameter $\lambda$ in the continuation procedure \Cref{alg:continuation}. The results of our implementations are reported in \Cref{t.l22l1} and the final outputs of the algorithms, along with the ground true Hubble image, are illustrated in \Cref{fig:hubbleData}.
 
 \vspace{-2mm}
\begin{table}[htbp]
\caption{A comparison among BPALM, A-BPALM1, and A-BPALM2 for the Hubble image, where the algorithms stopped after 90 seconds of running time. Here, $F_{error}$ stands for the relative fidelity error $\|X-U^kV^k\|_F/\|X\|_F$ while $O_{error}$ denotes the orthogonal error $\|I-V^k(V^k)^T\|_F$. In each row, the smallest number of $F_{error}$ and $F_{error}$ are displayed in bold.} 
\label{t.l22l1}
\begin{center}\footnotesize
\renewcommand{\arraystretch}{1.25}
\begin{tabular}{|l|ll|ll|ll|}
\hline
\multicolumn{1}{|c}{Penalty par.}                   
&\multicolumn{2}{|c}{BPALM}   
&\multicolumn{2}{|c}{A-BPALM1}   
&\multicolumn{2}{|c|}{A-BPALM2}  \\ 
\hline
$\lambda$ & $F_{error}$ & $O_{error}$ & $F_{error}$ & $O_{error}$ 
& $F_{error}$ & $O_{error}$\\ 
\hline
1 &  $6.75\times 10^{-2}$ & $8.28\times 10^{-2}$ & ${\bf 5.23\times 10^{-2}}$ & ${\bf 2.60\times 10^{-2}}$ & $5.32\times 10^{-2}$ & $3.25\times 10^{-2}$ \\
10 &  $1.54\times 10^{-1}$ & $4.35\times 10^{-2}$ & ${\bf 6.04\times 10^{-2}}$ & ${\bf 8.36\times 10^{-3}}$ & $9.33\times 10^{-2}$ & $2.16\times 10^{-2}$ \\
100 &  $2.05\times 10^{-1}$ & $3.28\times 10^{-2}$ & ${\bf 9.21\times 10^{-2}}$ & ${\bf 5.15\times 10^{-3}}$ & $1.99\times 10^{-1}$ & $1.01\times 10^{-2}$ \\
1000 &  ${\bf 2.07\times 10^{-1}}$ & $3.51\times 10^{-2}$ & $2.09\times 10^{-1}$ & ${\bf 2.51\times 10^{-3}}$ & $2.48\times 10^{-1}$ & $6.83\times 10^{-3}$ \\
10000 &  $2.09\times 10^{-1}$ & $3.44\times 10^{-2}$ & ${\bf 2.62\times 10^{-2}}$ & ${\bf 1.45\times 10^{-3}}$ & $2.55\times 10^{-1}$ & $6.33\times 10^{-3}$ \\
\hline
\end{tabular}
\end{center}
\end{table}

\begin{figure}[h!]
\centering
\begin{subfigure}[t]{0.99\textwidth}
\centering
{\includegraphics[width=12.5cm]{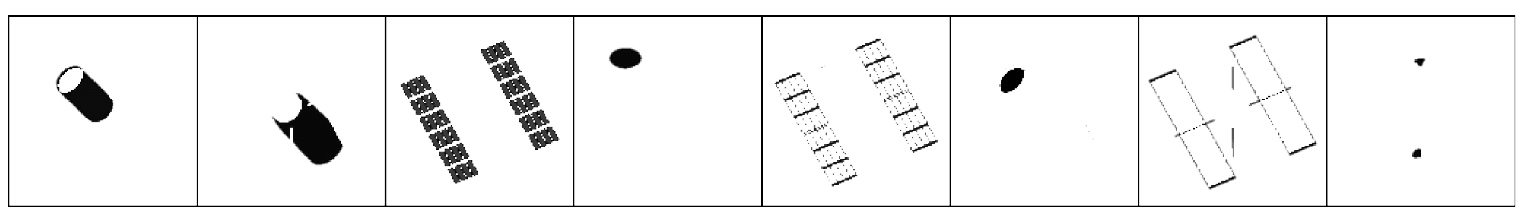}}%
\vspace{-1mm}
\caption{Ground truth image}
\end{subfigure}
\quad
\begin{subfigure}[t]{0.99\textwidth}
\centering
{\includegraphics[width=12.5cm]{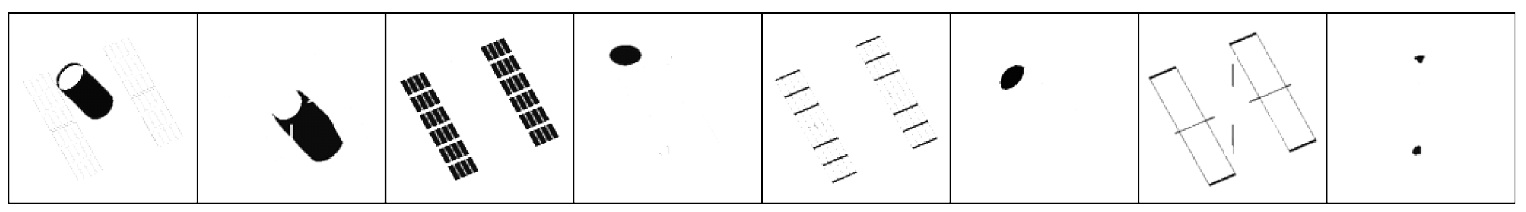}}%
\caption{BPALM,\ $F_{error}=1.54\times 10^{-1}$,\ $O_{error}=4.35\times 10^{-2}$}
\end{subfigure}
\quad
\begin{subfigure}[t]{0.99\textwidth}
\centering
{\includegraphics[width=12.5cm]{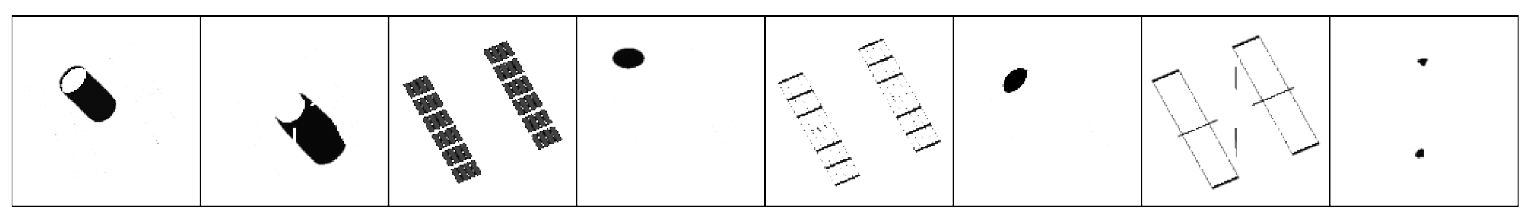}}%
\caption{A-BPALM1,\ $F_{error}=6.04\times 10^{-2}$,\ $O_{error}=8.36\times 10^{-3}$}
\end{subfigure}
\quad
\begin{subfigure}[t]{0.99\textwidth}
\centering
{\includegraphics[width=12.5cm]{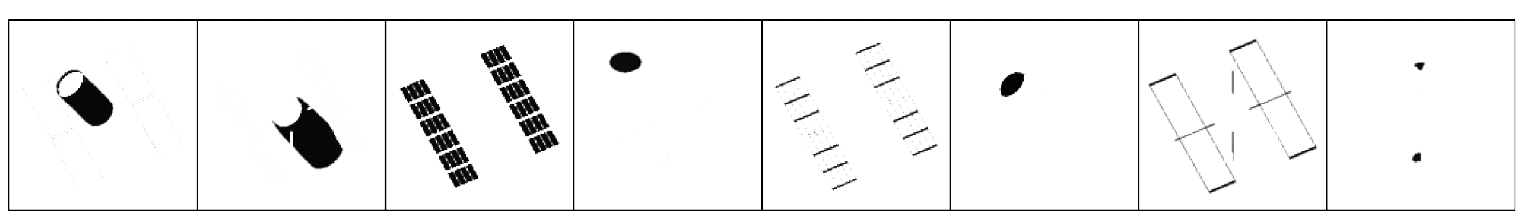}}%
\caption{A-BPALM2,\ $F_{error}=9.33\times 10^{-2}$,\ $O_{error}=2.16\times 10^{-2}$}
\end{subfigure}
 \vspace{-2mm}
\caption{A comparison among BPALM, A-BPALM1, and A-BPALM2 for the Hubble image, where the algorithms are stopped after 90 seconds of running time. Subfigure (a) shows the ground truth Hubble image, and Subfigures (b)-(d) are the results of BPALM, A-BPALM1, and A-BPALM2, respectively. 
In each subfigure, each image corresponds to a row of $V$ that has been reshaped as an image (since each entry corresponds to a pixel; see above).}
\label{fig:hubbleData}
\end{figure}

From \Cref{t.l22l1}, we observe that A-BPALM1 attains the better error than BPALM and A-BPALM2 in the sense of both the relative fidelity and orthogonal errors. Further, we observe that the orthogonal errors $O_{error}$ produced by the algorithms are decreasing by increasing the initial penalty parameter $\lambda$. From \Cref{fig:hubbleData}, A-BPALM1 provides slightly better quality image compared to BPALM and A-BPALM2 (look for example at the first basis image).

	\section{Conclusion}\label{sec:conclusion}
	We have analysed two new alternating linearized minimization algorithms called \refBPALM[] and \refaBPALM[] for solving the popular nonconvex nonsmooth optimization problem \eqref{eq:P}. Convergence analysis including the subsequential convergence, the global convergence and the convergence rate of the proposed algorithms is studied under the framework of multi-block relative smoothness and multi-block kernel functions. We emphasize that, to the best of our knowledge, \refBPALM[] and \refaBPALM[] are the first algorithms with rigorous convergence guarantee for solving ONMF in the literature.  We employ  \refBPALM[] and \refaBPALM[] to solve the orthogonal nonnegative matrix factorization problem. Some preliminary numerical tests are provided to illustrate the performance of our algorithms. 
A comprehensive numerical experiments with several data sets and comparison with state-of-the-art algorithms are out of the scope of the current paper, which we aim for future work.

%


	\vspace{-3mm}
	\ifsiam
		\bibliographystyle{siamplain}
	\else
		\bibliographystyle{plain}
	\fi
		\bibliography{TeX/Bibliography}

\end{document}